\documentclass[12pt]{article}
\usepackage{enumerate}
\usepackage{amsmath}
\usepackage{amsthm}
\usepackage{amsfonts}
\usepackage{amssymb}
\usepackage{xcolor}
\usepackage[numbers]{natbib}
\usepackage{graphicx,xcolor}
\graphicspath{ {Diagrams/} }
\usepackage{bbm} %%%%% For indicator function use

\theoremstyle{plain}
\newtheorem{question}{Question}

\newtheorem{problem}[question]{Problem}

\newtheorem{theorem}[question]{Theorem}
\newtheorem{proposition}[question]{Proposition}
\newtheorem{corollary}[question]{Corollary}

\newtheorem{remark}[question]{Remark}
\theoremstyle{definition}
\newtheorem{definition}[question]{Definition}
\newtheorem*{question*}{Question}

\numberwithin{question}{section}
\numberwithin{equation}{section}
\newtheorem*{theorem*}{Theorem}

\newcommand{\containers}{\mathcal{C}}

\newcommand{\PP}{\mathcal{P}}

\newcommand{\V}{\mathbf{V}}

\newcommand{\N}{\mathbb{N}}

\newcommand{\eps}{\varepsilon}

\DeclareMathOperator{\ex}{ex}
\DeclareMathOperator{\vol}{vol}

\DeclareMathOperator{\Boxes}{Box}
\DeclareMathOperator{\Ent}{Ent}
\DeclareMathOperator{\Forb}{Forb}

\newcommand{\norm}[1]{\lVert #1 \rVert}
\newcommand{\cutnorm}[1]{\norm{#1}_{\square}}
\newcommand{\deltacut}{\delta_{\square}}
\newcommand{\oi}{[0,1]}
\newcommand{\dcut}{d_{\square}}
\newcommand{\kcut}{\square_{k}}
\newcommand{\dkcut}{d_{\kcut}}
\newcommand{\deltakcut}{\delta_{\kcut}}

\newcommand{\limitsx}[1]{\widehat{#1}}

\title{Rectilinear approximation and volume estimates for hereditary bodies via $[0,1]$-decorated containers}
\author{Victor Falgas--Ravry\thanks{Institutionen f\"or Matematik och Matematisk Statistik, Ume{\aa} Universitet, 901 87 Ume{\aa}, Sweden. Email: \texttt{victor.falgas-ravry@umu.se}. Research supported by Swedish Research Council grant 2016-03488.} \and Robert Hancock\thanks{Institut f\"ur Informatik, Im Neuenheimer Feld 205, 69120 Heidelberg, Germany. Email: \texttt{hancock@informatik.uni-heidelberg.de}. The research leading to these results was partially supported by the
Deutsche Forschungsgemeinschaft (DFG, German Research Foundation) – 428212407.} \and Johanna Str\"omberg\thanks{Matematiska Institutionen, Uppsala Universitet, L\"agerhyddsv\"agen 1, 751 06 Uppsala, Sweden. Email: \texttt{johanna.stromberg@math.uu.se}} \and Andrew Uzzell\thanks{Email: andrew.uzzell@gmail.com}}
\begin{document}
\maketitle
\begin{abstract}	
We use the hypergraph container theory of Balogh--Morris--Samotij and Saxton--Thomason 	to obtain general rectilinear approximations and volume estimates for sequences of bodies closed under certain families of projections. We give a number of applications of our results, including a multicolour generalisation of a theorem of Hatami, Janson and Szegedy on the entropy of graph limits. Finally, we raise a number of questions on geometric and analytic approaches to containers.
\end{abstract}

\section{Introduction}

\subsection{Aims of the paper}
In a major breakthrough six years ago now, Balogh--Morris--Samotij~\cite{BaloghMorrisSamotij15} and Saxton--Thomason~\cite{SaxtonThomason15} developed powerful theories of hypergraph containers. Given a hypergraph $H$ satisfying some smoothness assumptions they showed that there exists a \emph{small} collection of \emph{almost independent} sets whose subsets \emph{contain} all the independent sets of $H$. A wide variety of problems in combinatorics are equivalent to estimating the number of independent sets in various hypergraphs; the groundbreaking work of~\cite{BaloghMorrisSamotij15,SaxtonThomason15} has thus seen an equally wide variety of applications, see e.g. the surveys~\cite{BaloghWagner16,BaloghMorrisSamotij18}.

In this paper, our aim is to explore the implications of container theory beyond the discrete setting (which has hitherto been the main focus in applications) to the continuous setting, and to ask whether it is possible to obtain some form of containers going in the other direction, i.e. starting from results in the continuous setting. We do this in two ways.

First of all, we relate hypergraph containers to rectilinear approximation of continuous bodies. Informally, we show that container theory implies the following: consider a sequence of bodies $(b_n)_{n \in \mathbb{N}}$, where $b_n \subseteq [0,1]^{d_n}$. Suppose this sequence is closed under certain projections (satisfying some simple, natural conditions). Then the bodies in the sequence can be finely approximated by a small number of boxes. This (informally stated) result, Theorem~\ref{theorem: containers for hereditary bodies}, allows us to apply container theory to functions from discrete structures to $[0,1]$--- for instance, we estimate (Theorem~\ref{thm: probability lipschitz}) the probability that a random function from the Boolean hypercube to $[0,1]$ is $c$-Lipschitz. We use Theorem~\ref{theorem: containers for hereditary bodies} to obtain general volume estimates for hereditary bodies. Similarly to applications of container theory to counting problems, this requires certain supersaturation results, which in much of the literature are obtained in an ad hoc manner. One of our contributions in this paper is to obtain a general, widely applicable form of supersaturation under a natural assumption (which is satisfied in most examples that have been studied), leading to a very clean general statement, Theorem~\ref{theorem: counting for homogeneous ssee}, for volume estimates. A key question arising from this part of the paper is whether our rectilinear approximation results could be obtained directly from purely geometric considerations: given a sequence of bodies, is being closed under some family of projections enough to ensure the existence of good rectilinear approximations without resorting to container machinery? Also could some (weak) form of container theorem be obtained from geometric approximation arguments?

Secondly, in what was the initial motivation of this work, we investigate links between hypergraph containers and the theory of graph limits. Via container theory, we prove a multicolour generalisation of a theorem of Hatami, Janson and Szegedy~\cite{HatamiJansonSzegedy18} on the entropy of graph limits (Theorem~\ref{theorem: maximum entropy gives growth rate for [k]}). Our work in this part of the paper leads us to two questions. Can one extend the Hatami--Janson--Szegedy theorem further to $[0,1]$-decorated graph limits? This connects to a broader project of Lov\'asz and Szegedy~\cite{LovaszSzegedy10} on extending the theory of graph limits to limits of compact decorated graphs. Further, as above, is it possible to go in the other direction, and to derive some (weak) form of container theorem for graph properties from compactness results for graphons?

We note our work in this paper focusses exclusively on `thick' hereditary bodies, whose volume varies exponentially with the dimension, rather than `thin' bodies whose volume decays superexponentially. It is thus natural to ask whether one can obtain a set of streamlined general results similar to the ones we derive in this paper but for `thin' bodies. Also, our work suggest families of new Tur\'an-type entropy maximisation problems. These, along with the questions raised above, are discussed in greater detail in Section~\ref{section: conclusion}.

\subsection{Background}
The problem of estimating the  number of members of a hereditary class of discrete objects and characterising their typical structure  has a long and distinguished history, beginning with the work of Erd{\H o}s, Kleitman and Rothschild~\cite{ErdosKleitmanRothschild76} in the 1970s. The Alekseev--Bollob\'as--Thomason theorem~\cite{Alekseev92,Alekseev93,BollobasThomason97} determined the asymptotics of the logarithm of the number of graphs on $n$ vertices in a hereditary property of graphs, while  Alon, Balogh, Bollob\'as and Morris~\cite{AlonBaloghBollobasMorris11} characterised their typical structure. Further Conlon and Gowers~\cite{ConlonGowers10} and Schacht~\cite{Schacht2009} obtained general transference results, which in particular implied sparse random analogues of extremal theorems for monotone properties of graphs (see the ICM survey of Conlon~\cite{Conlon14} devoted to this topic).

In a major development in 2015, Balogh, Morris and Samotij~\cite{BaloghMorrisSamotij15} and independently Saxton and Thomason~\cite{SaxtonThomason15} developed powerful theories of hypergraph containers. Informally, they showed that --- given some smoothness conditions --- one may find in an $r$-uniform hypergraph $H$ on $n$ vertices a \emph{small} (size $2^{o(n^2)}$) collection $\mathcal{C}$ of \emph{almost independent sets} (containing at most $o(n^r)$ edges), with the \emph{container} property that every independent set $I$ in $H$ is contained inside some $C\in \mathcal{C}$. Provided one has a good understanding of the size and structure of the largest independent sets in $H$ (which is an extremal problem), one can use containers to estimate the number of independent sets in $H$ and characterise their typical structure, and to transfer such results to sparse random subhypergraphs of $H$. Since many well-studied hereditary properties of discrete structures correspond to the collection of independent sets in some suitably defined hypergraphs, the ground-breaking work of Balogh--Morris--Samotij and Saxton--Thomason has had an enormous number of applications, providing new, simplified proofs of many previous results as well as the resolution of many old conjectures. In the six years elapsed since their publication, the papers~\cite{BaloghMorrisSamotij15, SaxtonThomason15} had amassed over 250 citations each. Among these let us note the works of greatest relevance to the present paper, namely the work of Balogh and Wagner~\cite{BaloghWagner16} showcasing the versatility of the container method, the papers of Terry~\cite{Terry16} and Falgas-Ravry--O'Connell--Uzzell~\cite{FalgasRavryOconnellUzzell18} on applications of containers to multicoloured discrete structures, and the ICM survey of Balogh, Morris and Samotij~\cite{BaloghMorrisSamotij18} devoted to hypergraph containers.

Following on~\cite{FalgasRavryOconnellUzzell18}, in which containers were adapted to the multicolour setting via random colouring models whose (discrete) entropy was used to count the underlying multicoloured structures, we shall in this paper use (continuous) entropy in combination with containers to estimate the volume of hereditary bodies. Entropy was introduced by Shannon~\cite{Shannon48} in a foundational paper on information theory; the use of entropy for counting (in the discrete setting) or making volume estimates (in the continuous setting) is a well-established technique in combinatorics, see e.g. the lectures of Galvin~\cite{Galvin14} on this topic. Mention should be made here of the recent and impressive results of Kozma, Meyerovitch, Peled and Samotij~\cite{KozmaMeyerovitchPeledSamotij21} who obtained a very fine approximation of the metric polytope (see the discussion in Section~\ref{subsection: metric polytopes}) via much more sophisticated and involved entropy techniques than the ones used in this paper.

One motivation for writing this paper was to better understand the potential links between container theory and the theory of graph limits. Giving an exposition of the latter theory is beyond the scope of this paper, and we refer the interested reader to the monograph of Lov\'asz on the topic~\cite{LovaszBook}. It suffices to say here that in the theory of (dense) graph limits one passes from the discrete world of graphs to the continuous world of graphons, which are symmetric measureable functions $W: \ [0,1]^2\rightarrow [0,1]$. One can then seek to recover many finitary results of graph theory in the limit world of graphons via analytic techniques (or in fact prove new results which can be exported back to the world of finite graphs). In a 2018 paper, Hatami, Janson and Szegedy~\cite{HatamiJansonSzegedy18} defined an entropy function for graphons, and used this entropy function to reformulate and give an alternative proof of the Alekseev--Bollob\'as--Thomason theorem in the graph limit setting. Lov\'asz and Szegedy~\cite{LovaszSzegedy10} began extending the theory of graph limits from ordinary graphs to graphs whose edges are decorated or coloured with elements from a compact set; further work in this direction was done more recently by Kunszenti-Kov\'acs, Lov\'asz and Szegedy~\cite{KKLS14}, though (as we mention in Section~\ref{section: conclusion}) some parts of the theory are yet to be extended, such as the extraction of convergent subsequences from an arbitrary sequence of decorated graphs.

\subsection{Definitions and statement of our main results}
Before we can state our main results, we need to introduce some basic notation, to recall definitions of set sequences equipped with embeddings (ssee-s) and to define a number of concepts related to $[0,1]$-decorations, entropy and  hereditary properties of decorated ssee-s. Throughout the paper we shall use $\vert A\vert $ to denote the Lebesgue measure of $A$ when $A$ is a measurable subset of $[0,1]$. We also let $[n]$ denote the discrete interval $\{1, 2, \ldots , n\}$.  Finally, we shall use standard Landau notation: given functions $f,g:\ \mathbb{N} \rightarrow \mathbb{R}^+$, we write $f=o(g)$ for $\lim_{n\rightarrow \infty} f(n)/g(n)=0$ and $f=O(g)$ if there exists a constant $C>0$ such that $\limsup_{n\rightarrow \infty}f(n)/g(n)\leq C$. Further we write $f=\omega(g)$ for $g=o(f)$, $f=\Omega(g)$ for $g=O(f)$. 
\subsubsection{Set sequences equipped with embeddings (ssee)}\label{subsection: definition of ssee}
We begin by recalling the definition of a \emph{set sequence equipped with embeddings (ssee)}, and that of a \emph{good ssee}, from the prequel~\cite{FalgasRavryOconnellUzzell18} to this paper.
\begin{definition}[ssee]\label{definition: ssee}
	A \emph{set-sequence equipped with embeddings}, or \emph{ssee}, is a sequence~$\V=(V_n)_{n\in \N}$ of sets $V_n$, together with for every $N\leq n$ a collection $\binom{V_n}{V_N}$ of injections~$\phi: \ V_N \to V_n$. We refer to the elements of $\binom{V_n}{V_N}$ as \emph{embeddings} of $V_N$ into $V_n$.
\end{definition}
\begin{definition}[Intersecting embeddings]
	Let $\V=(V_n)_{n \in \N}$ be a ssee. Let $N_1$,~$N_2 \leq n$.  An \emph{$i$-intersecting embedding} of~$(V_{N_1}, V_{N_2})$ into $V_n$ is a function $\phi: \ V_{N_1}\sqcup V_{N_2}\rightarrow V_n$ such that:
	\begin{enumerate}[(i)]
		\item the restriction of $\phi$ to  $V_{N_1}$ lies in $\binom{V_n}{V_{N_1}}$, and the restriction of $\phi$ to $V_{N_2}$ lies in $\binom{V_n}{V_{N_2}}$;
		\item $\vert \phi(V_{N_1})\cap \phi(V_{N_2})\vert=i$.
	\end{enumerate}
	We denote by $I_i\bigl( (V_{N_1}, V_{N_2}), V_n\bigr)$ the number of $i$-intersecting embeddings of $(V_{N_1}, V_{N_2})$ into $V_n$, and set
	\[I(N, n):=\sum_{1<i<\vert V_N\vert} I_i\bigl( (V_{N}, V_{N}), V_n\bigr).\]		
\end{definition}
\begin{definition}[Good ssee]\label{definition: good ssee}
	A ssee  $\V$ is \emph{good} if it satisfies the following conditions:
	\begin{enumerate}[(i)]
		\item $\vert V_n\vert \rightarrow \infty$ (`the sets in the sequence become large');
		\item for all $N \in \N$ with $\vert V_N\vert >1$,  $\bigl\vert \binom{V_n}{V_N} \bigr\vert \gg \vert V_n\vert $ (`on average, vertices in $V_n$ are contained in many embedded copies of $V_N$');
		\item for all $N \in \N$ with $\vert V_N\vert >1$, $\Bigl(\vert V_n\vert I(N,n)\Bigr)\Big/ \bigl\vert \binom{V_n}{V_N} \bigr\vert ^2 \rightarrow 0$ as $n\rightarrow \infty$ (`most pairs of embeddings of~$V_N$ into $V_n$ share at most one vertex').
	\end{enumerate}	
\end{definition}
\noindent The notion of a ssee covers a wide variety of well-studied structures: all of the following are examples of good ssee-s:
\begin{itemize}
\item $V_n$ is the edge-set of the complete graph on $n$ vertices $E(K_n)$, and $\binom{V_n}{V_N}$ is the collection of maps $E(K_N)\rightarrow E(K_n)$ corresponding to the collection of graph isomorphisms from $K_N$ into $K_n$;
\item $V_n$ is $\{0,1\}^n$, the vertex-set of the $n$-dimensional hypercube $Q_n$, and $\binom{V_n}{V_N}$ is the collection of graph isomorphisms from $Q_N$ into $Q_n$;
\item $V_n$ is $[n]$, the interval of the first $n$ natural numbers, and $\binom{V_n}{V_N}$ is the collection of all injections $\phi$ sending $[N]$ into an arithmetic progression of $[n]$ of length $N$, i.e. $\phi: x \mapsto a+xd$ where $a,d$ are fixed non-negative integers and $d>0$;
\item $V_n$ is ${\left(\mathbb{F}_p\right)}^n$, where $p$ is a prime and $\mathbb{F}_p$  is the finite field with $p$ elements, and $\binom{V_n}{V_N}$ is the collection of shifts of injective additive homomorphisms from ${\left(\mathbb{F}_p\right)}^N$ to  ${\left(\mathbb{F}_p\right)}^N$-subgroups of ${\left(\mathbb{F}_p\right)}^n$ (i.e. the collection of all maps $\phi: \  \mathbf{a} +\psi(\mathbf{x})$, where $\mathbf{a}\in {\left(\mathbb{F}_p\right)}^n$, $\psi: \  {\left(\mathbb{F}_p\right)}^N \rightarrow  {\left(\mathbb{F}_p\right)}^n$ is injective and satisfies $\psi(\mathbf{x}+\mathbf{y})=\psi(\mathbf{x})+\psi(\mathbf{y})$; the image of ${\left(\mathbb{F}_p\right)}^N$ under $\phi$ is thus a coset of an ${\left(\mathbb{F}_p\right)}^N$-subgroup of ${\left(\mathbb{F}_p\right)}^n$);
\item somewhat similar to the example above,  $V_n$ is ${\left(\mathbb{F}_2\right)}^n\setminus\{\mathbf{0}\}$, and $\binom{V_n}{V_N}$ is the set of linear isomorphisms from ${\left(\mathbb{F}_2\right)}^N$ to $n$-dimensional linear subspaces of ${\left(\mathbb{F}_2\right)}^n$, restricted to ${\left(\mathbb{F}_2\right)}^N\setminus\{\mathbf{0}\}$ (this example corresponds to \emph{simple binary matroids} $M: \ {\left(\mathbb{F}_2\right)}^n\setminus\{\mathbf{0}\}\rightarrow\{0,1\}$, whose hereditary properties were recently investigated by Grosser, Hatami, Nelson and Norin~\cite{GrosserHatamiNelsonNorin21});
\item $V_n$ is the collection $\mathcal{P}([n])$ of all subsets of $[n]$, viewed as a poset under the subset relation, and $\binom{V_n}{V_N}$ is the collection of all injective poset homomorphisms from $\mathcal{P}([N])$ into $\mathcal{P}([n])$ (these homorphisms can be counted using~\cite[Theorem 4.1]{FalgasRavryMarkstromTreglownZhao20}, from which properties (ii) and (iii) follow easily);

\end{itemize}

\subsubsection{$[0,1]$-decorations, entropy, hereditary properties}\label{definition of [0,1]-decorations, etc}
\noindent We now generalise a number of definitions from~\cite{FalgasRavryOconnellUzzell18} to the setting of $[0,1]$-decorated ssee-s. In what follows, we write $(v_i)_{i\in I}$ to denote a vector $v$ whose coordinates are labelled with elements of some index set $I$, and refer to such vectors as \emph{$I$-indexed vectors}. Given a set $S$, we also write $S^I$ for the collection of $I$-indexed vectors all of whose coordinates take values in $S$, i.e. for the $I$-indexed Cartesian product $\prod_{i\in I}S$.
\begin{definition}[{$[0,1]$}-decorated sets and set sequences]\label{definition: decorated set sequences}
	Given a set $V$, a \emph{$[0,1]$-decoration of~$V$} is an element $c\in[0,1]^V$, i.e.~a function $c: \ V\rightarrow [0,1]$.  Given a set-sequence $\V=(V_n)_{n\in \N}$, a \emph{$[0,1]$-decoration of~$\V$} is a sequence $\mathbf{c}=(c_n)_{n\in \N}$, where for each $n\in \N$, $c_n$ is a $[0,1]$-decoration of~$V_n$. If $\V$ is a ssee, we call $\mathbf{c}$ a \emph{$[0,1]$-decorated ssee}.
\end{definition}
\noindent A $[0,1]$-decorated set is just a function from the set into $[0,1]$. Our interest in this paper is the extent to which such hereditary families of such functions can be approximated by (rectilinear) \emph{boxes}, which we define below.
\begin{definition}[Boxes, cylinders]\label{definition: boxes, cylinders}
	A  \emph{box} in $V_n$ is a Cartesian product of the form
	\[b:=\prod_{i\in V_n} A_i,\]
	where for each $i$, $A_i$ is a measurable subset of $[0,1]$.  (Thus $b$ is a collection of $V_n$-indexed vectors.) We write $\Boxes(V_n)$ for the collection of all boxes in $V_n$.  A \emph{$d$-cylinder} is a box where all but at most $d$ of the $A_i$ are equal to $[0,1]$.

	 A cylinder or box is said to be \emph{simple}  if for every $i\in V_n$, its projection onto coordinate $i$ is a finite union of intervals.  Further, a simple box is called \emph{$k$-rational} if each of these intervals is of the form~$[\frac{x}{k}, \frac{y}{k}]$, for some integers $0\leq x<y\leq k$. 
\end{definition}
\begin{definition}[Volume, entropy, density]\label{definition: volume, entropy density}
	Given a measurable body $b\subseteq [0,1]^{V_n}$, we denote its Lebesgue measure by $\vol(b)$. So for instance if $b=\prod_{i\in V_n} A_i$ is a box, its volume is $\vol(b)=\prod_{i\in V_n} \vert A_i\vert$.

	We also consider the volume of lower-dimensional bodies obtained from $b$ by fixing the coordinates inside some subset $I\subseteq V_n$; we then denote the corresponding $\vert V_n\setminus I\vert$-dimensional volume (with respect to the Lebesgue measure) by $\vol_{V_n\setminus I}$. So for instance for $v_0\in V_n$ 
	\[\vol_{V_n \setminus\{ v_0\}}(\{c\in b: \ c_{v_0}=1/2\}) \]
	denotes the volume of the $(\vert V_n\vert-1)$-dimensional object obtained by taking the intersection of $b$ with the hyperplane $\{ x\in \mathbb{R}^{V_n}:\ x_{v_0}=1/2\}$.

	Further, for a measurable body $b\subseteq [0,1]^{V_n}$, we define the \emph{entropy} of $b$ as
	\[\Ent(b):= -\log \vol(b).\]
	Finally, the \emph{density} of $b$ is
	\[d(b):= \vol(b)^{1/\vert V_n\vert}. \] 
	Observe that this latter quantity is an element of $[0,1]$, and that $d(b)= e^{-\Ent(b)/\vert V_n\vert}$.
\end{definition}
\noindent We now turn to the problem of defining what we mean by a \emph{hereditary} family of functions on a ssee. To do this, we first define a notion of projection inherited from the embeddings associated to the ssee.
\begin{definition}[Projections, lifts, shadows]\label{definition: projections}
	Let $\V=(V_n)_{n \in \N}$ be a ssee. Given an embedding~$\phi\in \binom{V_n}{V_N}$ and  a measurable body $b\subseteq [0,1]^{V_n}$, we denote by $b_{\downarrow \phi}$ the measurable subset of $[0,1]^{V_N}$ given by 
	\[b_{\downarrow \phi} := \left \{ c\in [0,1]^{V_N} \ : \  \vol_{V_n\setminus \phi(V_N)} \left(  \left \{ \tilde{c}\in b \ : \ \tilde{c}_{\phi(i)}= c_i \ \forall i\in V_N \right \} \right) >0 \right \}.\]
	So for example, given a box $b=\prod_{i \in V_n} A_i$, we have 
	\[  b_{\downarrow \phi}:=\begin{cases} 
	\prod_{i \in V_N} A_{\phi(i)} & \text{ if } \vol(A_j)>0 \ \forall j \in V_n \setminus \phi(V_N); \\
	\prod_{i \in V_N} \emptyset & \text{ otherwise.}
	\end{cases} \] 
	We call $b_{\downarrow\phi}$ the (strict) \emph{$\phi$-projection} of $b$. 
		
	 Conversely, given an embedding~$\phi\in \binom{V_n}{V_N}$ and  a measurable body $b\subseteq [0,1]^{V_N}$, we denote by $b_{ \uparrow \phi}$ the measurable subset of $[0,1]^{V_n}$ induced by $\phi$, namely 
	\[b_{\uparrow \phi} := \left\{c\in [0,1]^{V_n}: \ \exists \tilde{c}\in b \ \text{ such that } \forall i\in V_N, \ \tilde{c}_{i}= c_{\phi(i)}\right\}.\]	
	So for example, given a box $b=\prod_{i \in V_N} A_i$, we have $b_{\uparrow \phi}:=\prod_{i \in V_n} B_i$, where 
	\[ B_i := \begin{cases}
	A_{\phi^{-1}(i)} & \text{ if } i \in \phi(V_N); \\
	[0,1] & \text{ otherwise.}
	\end{cases} \]
	We call $b_{\uparrow \phi}$ the \emph{$\phi$-lift} of $b$. 
		
	Given $b\in [0,1]^{V_n}$ and $N\leq n$, we define the \emph{lower shadow}  of $b$ in $[0,1]^{V_N}$ by
	\[\partial^-_{V_N}(b):= \bigcup_{\phi \in \binom{V_n}{V_N}} b_{\downarrow \phi}.\]	
	Similarly, given $b\in [0,1]^{V_N}$ and $n\geq N$, we define the \emph{upper shadow} of $b$ in $[0,1]^{V_n}$ by:
	\[\partial^+_{V_n}(b):= \bigcup_{\phi \in \binom{V_n}{V_N}} b_{\uparrow \phi}.\]
\end{definition}
\noindent 
Observe that if $b$ is a box in $V_n$ and $\phi\in \binom{V_n}{V_N}$, then $b_{\downarrow \phi}\in \Boxes(V_N)$.  Conversely, if $b$ is a box in $V_N$ and  $\phi\in \binom{V_n}{V_N}$, then $b_{\uparrow \phi}$ is a box (in fact a $\vert V_N\vert$-cylinder) in $V_n$.  Also, for any body $b\subseteq [0,1]^{V_N}$ and any embeddings $\phi \in \binom{V_n}{V_N}$ , $\psi \in \binom{V_N}{V_{n'}}$ we have the relations
\[  (b_{\uparrow \phi})_{\downarrow \phi}= b \qquad \textrm{ and } \qquad \vol\left(b\setminus(b_{\downarrow \psi})_{\uparrow \psi} \right)=0. \]
\begin{definition}[Properties, hereditary properties]\label{definition: [0,1]-decorated properties}
	Let $\V=(V_n)_{n \in \N}$ be a ssee. A \emph{$[0,1]$-decoration property of $\V$} is a sequence $\PP=(\PP_n)_{n\in \N}$, where $\PP_n$ is a measurable subset of $[0,1]^{V_n}$. A $[0,1]$-decoration property is \emph{hereditary} if for all $n\geq N$, the upper shadow of $\left([0,1]^{V_N}\setminus \PP_N\right)$ in $[0,1]^{V_n}$ is a subset of $\left([0,1]^{V_n}\setminus \PP_n\right)$.
	
\end{definition}
\noindent In other words, a $[0,1]$-decoration property~$\PP$ of $\V$ is hereditary if its complement is closed under taking upper shadows/$\phi$-lifts. In particular, this implies that $\PP$ itself is closed under taking lower shadows/$\phi$-projections. (Note however that the converse fails: a property being closed under taking lower shadows does \emph{not} imply its complement is closed under taking upper shadows.)

As an illustrative example, observe that our abstract definition above generalises the graph theoretic notion of a hereditary property. Indeed let $\left(\mathcal{G}_n\right)_{n\in \mathbb{N}}$ be a sequence of subgraphs of $K_n$, the complete graph on $n$ vertices, that is closed under taking induced subgraphs (for instance, one could take $\mathcal{G}_n$ to be the collection of all triangle-free subgraphs of $K_n$, or of all the subgraphs of $K_n$ containing no induced cycle of length $5$, say). Letting $V_n$ denote the edge-set of $K_n$, we can encode a graph $G\in \mathcal{G}_n$ as a box $b_G=\prod_{e \in E(K_n)}A_e$ by letting $A_e=[\frac{1}{2}, 1]$ if $e\in G$ and setting $A_e=[0,\frac{1}{2})$ otherwise. Letting $\mathcal{P}_n:=\bigcup_{G\in \mathcal{G}_n} b_G$, we have that the $[0,1]$-decoration property $\mathcal{P}_n$ of the ssee $\mathcal{V}=(V_n)_{n\in \mathcal{N}}$ is hereditary in the sense of Definition~\ref{definition: [0,1]-decorated properties}.

One class of hereditary $[0,1]$-decoration properties will be of particular interest to us in this paper.
\begin{definition}[Forbidden projections]\label{definition: forb(F)}
	Let $\V=(V_n)_{n \in \N}$ be a ssee.  Let $b\subseteq [0,1]^{V_N}$. For $n\geq N
	$, we say $b' \subseteq [0,1]^{V_n}$ is \emph{$b$-free} if 	$\partial_{V_n}^+(b)\cap b'=\emptyset$. We denote by $\Forb(b)$ the hereditary $[0,1]$-decoration property of being $b$-free, i.e.~for all $n\geq N$
	\[\Forb(b)_n= [0,1]^{V_n}\setminus\partial_{V_n}^+(b) .\]
\end{definition}

\begin{definition}[Extremal entropy]\label{definition: extremal entropy/relative to possee}
	Let $\V=(V_n)_{n \in \N}$ be a ssee, and let  $\PP=(\PP_n)_{n\in \N}$ be a $[0,1]$-decoration property of $\V$. The \emph{extremal entropy} of $\PP$ \emph{relative to $\V$} is
	\[\ex(\V, \PP)_n=\ex(V_n, \PP_n):=\inf\left\{\Ent(b) \,: \, b\in \Boxes(V_n), \ \vol(b\setminus \PP_n)=0\right\}.\]
\end{definition}
\noindent Thus  $\exp\left(-\ex(V_n, \PP_n)\right)$ is precisely the volume of the largest box in $V_n$ which (up to a zero-measure set) is contained inside $P_n$. 

\subsubsection{Rectilinear approximation and volume estimates for hereditary bodies}\label{subsection: container theorems}

\noindent Let $\V$ be a good ssee and let $n,N \in \mathbb{N}$. Given  a function $f:\ \binom{V_n}{V_N}\rightarrow \mathbb{R}$, we write $\mathbb{E}_{\phi} f(\phi)$ for the expected value of $f(\phi)$ over $\phi \in \binom{V_n}{V_N}$ chosen uniformly at random. 

Our first result is a geometric approximation property for a hereditary body.
\begin{theorem}\label{theorem: containers for Forb(finite union of boxes)}
	Let $\V$ be a good ssee.  Let $\mathcal{F}$ be a nonempty finite family of simple boxes in $[0,1]^{V_N}$, for some $N \in \N$. Set $b=\bigcup_{b_f\in \mathcal{F}}b_f$ and $\PP=\Forb(b)$. Then for every $\varepsilon>0$, there exists $n_0>0$ such that for any $n\geq n_0$ there exists a collection $\containers$ of simple boxes in $[0,1]^{V_n}$ satisfying:
	\begin{enumerate}[(i)]
		\item $\PP_n \subseteq \bigcup_{c\in \containers} c$;
		\item for every $c\in \containers$, $\mathbb{E}_{\phi} \vol\left(c_{\downarrow \phi}\cap b\right)<\varepsilon$; %(where the expectation is taken over $\phi \in \binom{V_n}{V_N}$);
		
		\item $\vert \containers\vert\leq  e^{\varepsilon \vert V_n\vert}$.
	\end{enumerate}		
\end{theorem}
\noindent In other words there exists a \emph{small} (property (iii)) collection of simple boxes such that their union \emph{contains} the body $\PP_n$ (property (i)). Further, each of them has a lower shadow \emph{almost disjoint} from $b$ (property (ii)) --- so these boxes ``almost'' lie in $\PP_n$. The union of these boxes is thus a ``good'' approximation for $\PP_n$.

 It is worth noting that of course any measurable body can be finely approximated by a collection of simple boxes --- the power of the container theory of Balogh--Morris--Samotij and Saxton--Thomason is the bound (iii) they give on the number of boxes required.

Building on Theorem~\ref{theorem: containers for Forb(finite union of boxes)}, we prove:
\begin{theorem}\label{theorem: containers for hereditary bodies}
	Let $\V$ be a good ssee.  Let $\PP$ be a hereditary property of $[0,1]$-decorations of $\V$ and let $N\in \mathbb{N}$. Then for every $\varepsilon>0$, there exists an integer
	 $n_0>N$ such that for any $n\geq n_0$ there exists a collection~$\containers$ of simple boxes in $[0,1]^{V_n}$ satisfying:
	\begin{enumerate}[(i)]
		\item $\vol \left(\PP_n \setminus \bigcup_{c\in \containers} c\right)=0$;
		\item for every $c\in \containers$, $\mathbb{E}_{\phi} \vol\left(c_{\downarrow \phi}\setminus \PP_N\right)<\varepsilon$;	%(where the expectation is taken over all $\phi \in \binom{V_n}{V_N}$);
		\item $\vert \containers\vert\leq  e^{\varepsilon \vert V_n\vert}$.
	\end{enumerate}		
\end{theorem}
\noindent We remark that we cannot replace (i) by a containment condition $\PP_n \subseteq \bigcup_{c\in \containers}c$, and that we must allow for an exceptional zero-measure set not covered by the simple boxes in $\containers$. For instance suppose $\PP$ consisted of all decorations $x\in [0,1]^{V_n}$ with $x_i \in [0, \frac{1}{2}]\cup \mathbb{Q}$. Then clearly we cannot both cover all of $\PP_n$ with simple boxes and still achieve (ii).

Observe also that condition (ii) may be interpreted as follows: suppose we take a point $\mathbf{x}$ chosen uniformly at random from $c$ and an embedding $\phi$ uniformly at random from $\binom{V_n}{V_N}$. This defines a random point $\mathbf{y}\in [0,1]^{V_N}$, by setting $y_v= x_{\phi(v)}$ for all $v\in V_N$. Then condition (ii) is saying that the probability $\mathbf{y}$ fails to be in $\mathcal{P}_N$ is small (at most $\varepsilon$).

Provided  we have (a) a limiting density and (b) supersaturation for a $[0,1]$-decoration property in a good ssee, the container/geometric approximation theorem, Theorem~\ref{theorem: containers for hereditary bodies}, immediately implies a volume estimate for the $V_n$-dimensional body $\PP_n$, namely:
\begin{corollary}\label{corollary: counting for ssee}
	Let $\V$ be a good ssee. Let $\PP$ be a 	hereditary property of $[0,1]$-decorations of $\V$. Suppose in addition that the following hold:
	\begin{enumerate}[(a)]
		\item $\pi(\PP):=\lim_{n\rightarrow \infty}\ex(V_n, \mathcal{{P}}_n)/\vert V_n\vert$ exists;
		\item for all $\varepsilon>0$, there exist $\eta>0$ and positive integers $N\leq n_0$ such that if $n\geq n_0$ then for every $b\in \Boxes(V_n)$, if $\mathbb{E}_{\phi} \vol\left(b_{\downarrow \phi}\setminus \PP_N\right) <\eta$
		%$\vol\left( \partial^+_{V_n} \left([0,1]^{V_N}\setminus \PP_N\right)\cap b \right)< \eta$ 
		then $\Ent(b)> \left (\pi(\PP)-\varepsilon \right)\vert V_n\vert$.
	\end{enumerate}	
	Then
	\[\vol\left(\PP_n\right)= e^{-\bigl(\pi(\PP)+o(1)\bigr)\vert V_n\vert }.\]
\end{corollary}

A natural question is whether we can give a simple criterion for satisfying assumptions (a) and~(b). In past applications, this has mostly been dealt with in an ad hoc manner. For example in~\cite{FalgasRavryOconnellUzzell18}, it was proved that these two assumptions were satisfied for vertex $k$-colourings  and edge $k$-colourings of both complete hypergraphs and hypercube graphs, but each case required a separate proof. One of our contributions in this paper is a very simple criterion on the family of embeddings which is sufficient to ensure (a) and (b) are satisfied. This criterion immediately applies to a wide class of structures, yielding volume estimate results of great generality.
\begin{definition}
	A ssee $\V$ is \emph{homogeneous} if for every $n\geq N$, every $x\in V_n$ is contained in the same, strictly positive number of embeddings $\phi(V_N)$, $\phi \in \binom{V_n}{v_N}$.
\end{definition}
\begin{proposition}\label{proposition: existence of limits from uniformity}
		Suppose $\V$ is a good homogeneous ssee. Then for every hereditary property $\PP$ of $[0,1]$-decorations of $\V$,  the sequence $\frac{\ex(V_n, \PP_n)}{\vert V_n\vert}$ is non-decreasing in $\mathbb{R}_{\geq 0}$. In particular, this sequence either converges to a limit~$\pi(\PP)$ or tends to infinity as $n\rightarrow \infty$.
\end{proposition}
\begin{theorem}\label{theorem: counting for homogeneous ssee}
		Suppose $\V$ is a good homogeneous ssee. Then for every hereditary property $\PP$ of $[0,1]$-decorations of $\V$, either 
		\[\lim_{n\rightarrow \infty}\frac{\ex(V_n, \PP_n)}{\vert V_n\vert} =\pi(\PP) \]
		exists and
		\[\vol(\PP_n)= e^{-\left(\pi(\PP) +o(1)\right)\vert V_n\vert } \]
		or $\ex(V_n, \PP_n)/\vert V_n\vert\rightarrow \infty$ as $n\rightarrow \infty$ and the volume of $\PP_n$ decays superexponentially in $\vert V_n\vert$, $\vol(\PP_n)\leq   e^{-\omega\left(\vert V_n\vert\right)}$.
	\end{theorem}

\subsection{Structure of the paper}
We prove our main results in Section~\ref{section: proofs}.  In Section~\ref{section: applications} we give simple applications of our results to Lipschitz functions on the hypercube, the metric polytope, and weighted graphs, while in Section~\ref{section: graph limits} we use our results to derive a multicolour/decorated version of an `entropy of graph limits' theorem of Hatami, Janson and Szegedy~\cite{HatamiJansonSzegedy18} (which was one of our initial motivations for undertaking this project). We end the paper in Section~\ref{section: conclusion} with a discussion of various questions arising from our work.

\section{Proofs of the theoretical results}\label{section: proofs}
For $k\in \N$, let $[k]$ denote the set $\{1,2,\ldots,k\}$. One of the main results of the prequel to this paper~\cite{FalgasRavryOconnellUzzell18} was a container theorem for $[k]$-decorated ssee-s (which itself was obtained as a consequence of a simple container theorem of Saxton and Thomason~\cite{SaxtonThomason16}). To state it we must recall a few definitions.

A $k$-colouring \emph{template} for $V_n$ is a function $t$ assigning to each element of $V_n$ a nonempty subset of $[k]$. A \emph{realisation} of $t$ is a colouring $c: \ V_n \rightarrow [k]$ with the property that $c(x)\in t(x)$ for all $x\in V_n$. The set of all realisations of $t$ is denoted by $\langle t\rangle$. 
Given a collection $\mathcal{F}$ of $k$-colourings of $V_N$, denote by $\Forb_{\V}(\mathcal{F})$ the hereditary property of $k$-colourings of $\V$ not containing an embedding of a colouring in $\mathcal{F}$, i.e.
\begin{align}\label{eq:forbdiscrete}
 \Forb_{\V}(\mathcal{F})_n := \left \{ c \in [k]^{V_n} : \forall \phi \in \binom{V_n}{V_N}, c_{\downarrow \phi} \not \in \mathcal{F} \right \},
 \end{align}
where in the $[k]$-decorated context $c_{\downarrow \phi} $ denotes the $k$-colouring of $V_N$ induced by $\phi$,  $c_{\downarrow \phi}: v\mapsto c(\phi(v))$.
\begin{theorem}[Theorem~3.18 in~\cite{FalgasRavryOconnellUzzell18}]\label{theorem: ssee container}
	Let $\V$ be a good ssee, and let $k$,~$N \in \N$. Let $\mathcal{F}$ be a nonempty subset of $[k]^{V_N}$ and $\PP=\Forb_{\V}(\mathcal{F})$.  For any $\varepsilon>0$, there exists $n_0>0$ such that for any $n\geq n_0$ there exists a collection $\mathcal{T}_n$ of $k$-colouring templates for $V_n$ satisfying:
	\begin{enumerate}[(a)]
		\item $\PP_n \subseteq \bigcup_{t\in \mathcal{T}_n} \langle t \rangle$;
		\item for each template $t\in\mathcal{T}_n$, there are at most $\varepsilon \bigl\vert \binom{V_n}{V_N} \bigr\vert$ pairs $(\phi,c)$ with $\phi \in \binom{V_n}{V_N}$, $c\in \mathcal{F}$ and $c\in \langle t_{\downarrow \phi}\rangle$;
		\item $\vert \mathcal{T}_n\vert\leq  k^{\varepsilon \vert V_n\vert}$.
	\end{enumerate}	
\end{theorem} 

Our strategy to prove the theorems in Section~\ref{subsection: container theorems} is to use compactness: as $[0,1]$ is bounded, we can approximate measurable sets by finite unions of intervals of the form $[\frac{i-1}{k}, \frac{i}{k}]$ where $k\in \N$ is some large constant and $i\in [k]$. There we can apply the container theorem for colourings by discrete finite sets, Theorem~\ref{theorem: ssee container}. Provided we are careful with our approximations (and, crucially, that we are using the right definitions), we are able to transfer the container results from the discrete to the continuous setting.

\begin{proof}[Proof of Theorem~\ref{theorem: containers for Forb(finite union of boxes)}]
	Fix $\varepsilon>0$. There exists $k\in \N$ such that there exists a finite union of $k$-rational simple boxes $\tilde{b}$ such that (1) $\tilde{b} \subseteq b$ and (2) $\vol(b\setminus \tilde{b})< \varepsilon/2$. We can now pass to the discrete setting and consider $[k]$-colourings of $V_N$ as a proxy for $k$-rational simple boxes in $[0,1]^{V_N}$. 
	Let
	\[\mathcal{F}:=\left\{c\in [k]^{V_N}:  \  \prod_{v\in V_N}\biggl[\frac{c_v -1}{k}, \frac{c_v}{k}\biggr] \subseteq \tilde{b}\right\}\]
	be the family of $k$-colourings of $V_N$ corresponding to $\tilde{b}$. Apply Theorem~\ref{theorem: ssee container} to $\mathcal{F}$ with parameters $k, N$ and $\varepsilon'= \min \left(\varepsilon/\log k, \varepsilon/2\right)$. Let ${n}_0\in \N$ be such that for all $n\geq  n_0$ conclusions (a)--(c) from Theorem~\ref{theorem: ssee container} hold. Let $\mathcal{T}_n$ be the family of templates whose existence is guaranteed by the theorem. For each $t\in \mathcal{T}_n$, we define a box $c^t$ given by
	\[c^t:= \prod_{v\in V_n} \left(\bigcup_{i\in t(v)}\biggl[\frac{i -1}{k}, \frac{i}{k}\biggr]\right).\]
	Let $\containers$ denote the collection of boxes thus obtained. Property (a) from Theorem~\ref{theorem: ssee container} implies
	\begin{align}\label{eq: C satisfies property i}
	\PP_n =  \left([0,1]^{V_n}\setminus \left(\partial^+_{V_n}(b)\right)\right) \subseteq \left([0,1]^{V_n}\setminus \left(\partial^+_{V_n}(\tilde{b})\right)\right)\subseteq \bigcup_{t\in \mathcal{T}_n} c^t= \bigcup_{c\in \containers}c.
	\end{align}
	Note that the second containment relation in~\eqref{eq: C satisfies property i} follows since by definition, the (continuous) upper shadow $\partial^+_{V_n}(\tilde{b})$ when $\tilde{b}$ is a union of $k$-rational simple boxes corresponds precisely to the (discrete set of) $c \in [k]^{V_n}$ which are not in $\Forb_{\V}(\mathcal{F})_n$ (see~(\ref{eq:forbdiscrete})). This is what motivated our choice of $\mathcal{F}$ above.

	Further, property (b) entails that if one fixes $t\in \mathcal{T}_n$ and picks $\phi\in \binom{V_n}{V_M}$ uniformly at random, there is at most an $\varepsilon'$-chance that $\left(c^t\right)_{\downarrow \phi}$ intersects the box ${\tilde{b}}$ in a set with non-zero measure. Thus
	\begin{align}
	\mathbb{E}_{\phi} \vol\bigl((c^t)_{\downarrow \phi}\cap b\bigr)&= \mathbb{E}_{\phi} \vol\left((c^t)_{\downarrow \phi}\cap(b\setminus \tilde{b})\right)+ \mathbb{E}_{\phi} \vol\left((c^t)_{\downarrow \phi}   \cap {\tilde{b}}\right)\leq \vol(b\setminus \tilde{b}) + \varepsilon'<\varepsilon.\label{eq: C satisfies property ii}
	\end{align}
	Finally, property (c) gives
	\begin{align}\label{eq: C satisfies property iii}
	\vert \containers\vert = \vert \mathcal{T}_n\vert \leq k^{\varepsilon' \vert V_n\vert}\leq e^{\varepsilon \vert V_n\vert}.\end{align}
	Together, (\ref{eq: C satisfies property i}), (\ref{eq: C satisfies property ii}) and (\ref{eq: C satisfies property iii}) show properties (i)---(iii) in the statement of Theorem~\ref{theorem: containers for Forb(finite union of boxes)} are satisfied as claimed, concluding the proof.
\end{proof}

\begin{proof}[Proof of Theorem~\ref{theorem: containers for hereditary bodies}]
	Fix $\varepsilon>0$.  Let $b:=[0,1]^{V_N}\setminus \PP_N$. Since $\PP$ is a measurable set, there exists a finite union of simple boxes~$\tilde{b}$ such that (1) $\vol(\tilde{b}\setminus b)=0$ (i.e.~up to a zero-measure set, $\tilde{b}\subseteq b$) and (2) $\vol(b\setminus \tilde{b})< \varepsilon/2$, by the definition of the Lebesgue measure.  Apply Theorem~\ref{theorem: containers for Forb(finite union of boxes)} to $\mathcal{Q}:=\Forb(\tilde{b})$ with parameter $\varepsilon/2$, and let $\containers$ be the resulting family of simple boxes.

As $\PP$ is hereditary we have $\partial^+_{V_n}(b) \subseteq [0,1]^{V_n} \setminus \PP_n$, which implies that
\[ \vol (\PP_n \setminus \mathcal{Q}_n) = \vol (\PP_n \cap \partial^+_{V_n}(\tilde{b})) \leq \vol(\PP_n \cap \partial^+_{V_n}(b)) + \vol( \partial^+_{V_n}(\tilde{b}) \setminus \partial^+_{V_n}(b))=0. \]
It follows that
\[ \vol\left(\PP_n\setminus \bigcup_{c\in \containers}c\right) \leq \vol\left(\PP_n\setminus \mathcal{Q}_n\right)+ \vol\left(\mathcal{Q}_n\setminus \bigcup_{c\in \containers}c\right) =0.\]
Further, we have 
\[ \vol\left(\mathcal{Q}_N \setminus \PP_N \right)= \vol\left([0,1]^{V_N} \setminus (\partial^+_{V_N}(\tilde{b}) \cup \PP_N) \right)=\vol\left(b \setminus  \tilde{b}\right)<\frac{\varepsilon}{2},\]
and so for every $c \in \mathcal{C}$ we have 
\[ \mathbb{E}_{\phi} \vol\left(c_{\downarrow \phi} \setminus  \PP_N\right) \leq \mathbb{E}_{\phi} \bigl(\vol\left(c_{\downarrow \phi}\setminus \mathcal{Q}_N\right)+ \vol(\mathcal{Q}_N\setminus \PP_N) \Bigr)< \frac{\varepsilon}{2}+\frac{\varepsilon}{2}= \varepsilon. \]

	Finally, $\vert \containers_n\vert \leq e^{\varepsilon \vert V_n\vert /2}$. Thus $\containers$ satisfies the properties (i)--(iii) claimed by Theorem~\ref{theorem: containers for hereditary bodies}, as desired.
\end{proof}
\begin{proof}[Proof of Corollary~\ref{corollary: counting for ssee}]
	Fix $\varepsilon>0$. Let $\eta>0$ and $N, n_0\in \N$ be the constants guaranteed by  assumption (b). Applying Theorem~\ref{theorem: containers for hereditary bodies} to $\PP$ with parameter $\delta= \min(\varepsilon, \eta)$, we find $n_1\geq n_0$ such that for all $n\geq n_1$ there exists a collection $\containers$ of simple boxes in $[0,1]^{V_n}$ such that (i) up to a zero-volume set, $\PP_n$ is contained inside $\bigcup_{c\in \containers} c$, (ii) for every $c\in \containers$,
	$\mathbb{E}_{\phi} \vol\left(c_{\downarrow \phi}\setminus \PP_N\right)<\delta$
	and (iii) $\vert \containers\vert \leq e^{\delta \vert V_n\vert}$.

	By assumption (b) and our choice of $\delta$ and $n_1$, this implies that for every $c\in \containers$ we have $\Ent(c)\geq \left(\pi(\PP) - \varepsilon\right) \vert V_n \vert$. Now (i) and (iii) allow us to bound the volume of $\PP_n$:
	\[
	\vol\left(\PP_n\right) \leq \vol \left( \bigcup_{c \in \mathcal{C}} c \right)= \sum_{c\in \containers} e^{-\Ent(c)} \leq \vert \containers\vert e^{-\left( \pi(\PP) - \varepsilon\right) \vert V_n\vert}\leq e^{\delta\vert V_n\vert } e^{-\left( \pi(\PP) - \varepsilon\right)\vert V_n\vert } \leq e^{-\left(\pi(\PP) - 2\varepsilon\right)\vert V_n\vert }.
	\]
	Since $\varepsilon>0$ was arbitrary, the theorem follows. 	
\end{proof}
\begin{proof}[Proof of Proposition~\ref{proposition: existence of limits from uniformity}]
	Assume that $\V$ is a good, homogeneous ssee, and let $\PP$ be a hereditary property of $[0,1]$-decorations of $\V$. Set $x_n= \frac{\ex(V_n, \PP_n)}{\vert V_n\vert}$. Let $b\in \Boxes(V_{n+1})$ be a box with $\vol\left(b\setminus \PP_{n+1}\right)=0$ and $\Ent(b)= \ex(V_{n+1}, \PP_{n+1})$.  
By homogeneity of the ssee $\V$,  each coordinate in $V_{n+1}$ is counted in the same nonzero number $k= \left| \binom{V_{n+1}}{V_{n}} \right| \frac{\vert V_n\vert }{\vert V_{n+1}\vert }$ of projections $\phi(V_n)$, so that the family $\{\phi(V_n): \ \phi \in \binom{V_n+1}{V_n}\}$ constitutes a $k$-uniform cover of $V_{n+1}$. Since $b$ is a box, it follows that
\begin{align}\label{eq: bounding volume by projections}
\vol(b)^{\left| \binom{V_{n+1}}{V_n} \right| \frac{\vert V_n\vert}{\vert V_{n+1}\vert}} 
= \prod_{\phi \in \binom{V_{n+1}}{V_n}} \vol({b}_{\downarrow \phi}).
\end{align}
Since $\PP$ is hereditary and $\vol(b\setminus \PP_{n+1})=0$, for every $\phi \in \binom{V_{n+1}}{V_n}$ the $\phi$-projection ${b}_{\downarrow \phi}$ is a box in $V_n$ satisfying $\vol(b_{\downarrow \phi}\setminus \PP_n)=0$. In particular, we must have
\begin{align*}
\vol({b}_{\downarrow \phi})\leq e^{-\ex(V_n, \PP)}= e^{-x_n \vert V_n\vert}.
\end{align*}
Combining this with (\ref{eq: bounding volume by projections}), we have
\begin{align*}
e^{-x_{n+1} \left| \binom{V_{n+1}}{V_n} \right| \vert V_n\vert }=\vol(b)^{\left| \binom{V_{n+1}}{V_n} \right| \frac{\vert V_n\vert}{\vert V_{n+1}\vert}}\leq e^{-x_n \left| \binom{V_{n+1}}{V_n} \right| \vert V_n\vert},
\end{align*}
implying $x_n\leq x_{n+1}$ as desired.	
\end{proof}
\begin{proof}[Proof of Theorem~\ref{theorem: counting for homogeneous ssee}]
	Assume that $\V$ is a good, homogeneous ssee, and let $\PP$ be a hereditary property of $[0,1]$-decorations of $\V$. Set $x_n= \frac{\ex(V_n, \PP_n)}{\vert V_n\vert}$.

	Suppose first of all that $x_n\rightarrow \pi(\PP)$ as $n\rightarrow\infty$. Thus assumption (a) from Corollary~\ref{corollary: counting for ssee} is satisfied. We show that assumption (b) is satisfied as well, whence our claimed volume estimate for $\PP_n$ is immediate. Fix $\varepsilon>0$. We may assume that $\pi(\PP)> \varepsilon$, else we have nothing to show. Now, by the monotonicity of $(x_n)_{n\in \mathbb{N}}$ established in Proposition~\ref{proposition: existence of limits from uniformity}, there exists a constant $N\in \N$ such that $x_{N}> \pi(\PP)-\frac{\varepsilon}{3}$. In particular there exists $\delta_1=\delta_1(N, \varepsilon)>0$ such that if $a\in \Boxes(V_N)$ satisfies $\Ent(a)\leq (\pi(\PP)-\frac{2\varepsilon}{3})|V_N|$, then we have $\vol\left(a\setminus \PP_N\right)>\delta_1$.

	Consider a box $b\in \Boxes(V_n)$ with $\Ent(b)\leq (\pi(\PP)-\varepsilon)|V_n|$ for some $n\geq N$.  Let $B$ be the family of $\phi\in \binom{V_n}{V_N}$ such that $\Ent(b_{\downarrow \phi})\leq (\pi(\PP)-\frac{2\varepsilon}{3})|V_N|$. By our observation in the previous paragraph, $\phi\in B$ implies $\vol(b_{\downarrow \phi}\setminus \mathcal{P}_N)>\delta_1$.  Now, by homogeneity of $\V$ and the fact the volume of a box is the product of its projections, we have
	\begin{align*}
	\left| \binom{V_n}{V_N} \right| \vert V_{N}\vert  \left(\pi(\PP)-\varepsilon\right)\geq \left| \binom{V_n}{V_N} \right| \frac{\vert V_{N}\vert}{\vert V_n \vert}\Ent(b)  &=\sum_{\phi \in \binom{V_n}{V_N}} \Ent(b_{\downarrow\phi}) \geq  \left\vert \binom{V_n}{V_N}\setminus B\right\vert \vert V_N \vert \biggl(\pi(\PP)-\frac{2\varepsilon}{3}\biggr),
	\end{align*}
	implying the existence of a constant $\delta_2=\delta_2(\varepsilon, \PP)>0$ such that $\vert B \vert \geq \delta_2 \left| \binom{V_{n}}{V_N} \right|$. (Explicitly, $\delta_2= \frac{\varepsilon}{3\pi(\PP)-2\varepsilon}$ will do.) It follows that
	\begin{align*}
	\mathbb{E}_{\phi} \vol(b_{\downarrow \phi} \setminus \PP_N)\geq \delta_1 \mathbb{P}\left(\phi \in B\right) \geq \delta_1\delta_2.
	\end{align*}
	Setting $\eta=\delta_1\delta_2$, we have that assumption (b) from Corollary~\ref{corollary: counting for ssee} is satisfied, and we are done in this case.

	Suppose now instead that $x_n\rightarrow \infty$ as $n\rightarrow\infty$. For every $C>0$, there exists a constant $N=N(C, \PP)\in \N$ such that $x_{N}>2C+2$. In particular there exists $\delta_1=\delta_1(C)>0$ such that if $a\in \Boxes(V_{N})$ satisfies $\Ent(a)\leq (2C+1)|V_{N}|$, then we have $\vol\left(a\setminus \PP_{N}\right)>\delta_1$.

	Consider a box $b\in \Boxes(V_n)$ with $\Ent(b)\leq 2C |V_n|$ for some $n\geq N$.  Let $B$ be the family of $\phi\in \binom{V_n}{V_N}$ such that $\Ent(b_{\downarrow \phi})\leq (2C+1)|V_N|$.  By homogeneity of $\V$ and the fact the volume of a box is the product of its projections, we have
	\begin{align*}
	\left| \binom{V_n}{V_N} \right| \vert V_{N}\vert  2C\geq \left| \binom{V_n}{V_N}\right| \frac{\vert V_{N}\vert}{\vert V_n \vert}\Ent(b)  &=\sum_{\phi \in \binom{V_n}{V_N}} \Ent(b_{\downarrow\phi}) \geq  \left\vert \binom{V_n}{V_N}\setminus B\right\vert \vert V_N \vert \biggl(2C+1\biggr),
	\end{align*}
	implying the existence of a constant $\delta_2=\delta_2(C)>0$ such that $\vert B \vert \geq \delta_2 \left| \binom{V_{n}}{V_N} \right|$. (Explicitly, $\delta_2= \frac{1}{2C+1}$ will do.) Now we have (by our observation in the paragraph above)
	\begin{align*}
	\mathbb{E}_{\phi} \vol(b_{\downarrow \phi} \setminus \PP_N)\geq \delta_1 \mathbb{P}\left(\phi \in B\right) \geq \delta_1\delta_2.
	\end{align*}
Thus we have shown the following:  for all $n\ge N$, and all boxes $b\in \Boxes(V_n)$, 	$\mathbb{E}_{\phi} \vol(b_{\downarrow \phi} \setminus \PP_N)<\delta_1\delta_2$ implies $\mathrm{Ent}(b)> 2C |V_n|$ $(\dagger)$.

	We now apply Theorem~\ref{theorem: containers for hereditary bodies} to the good ssee $\V$ and the hereditary property $\mathcal{P}$ with parameters $N\in \mathbb{N}$ and $\varepsilon=\varepsilon(C)=\min \left(C, \delta_1\delta_2\right)>0$: there exists $n_0=n_0(C)>N$ such that for all $n\geq n_0$ there exists a collection $\mathcal{C}\subseteq \Boxes(V_n)$ satisfying properties (i)---(iii) from the statement of Theorem~\ref{theorem: containers for hereditary bodies}. By $(\dagger)$ established above and our choice of $\varepsilon>0$, property (ii) implies that for every $c\in \mathcal{C}$,  $\mathrm{Ent}(c)>2C |V_n|$. Then properties (i), (iii)  and our choice of $\varepsilon\leq C$ together yield that for all $n\geq n_0(C)$,
	\[ \vol(\PP_n) \leq \vol \left( \bigcup_{c \in \mathcal{C}} c \right)  \leq \sum_{c \in \mathcal{C}} e^{-\Ent(c)} \leq |\mathcal{C}| e^{-2C|V_n|} \leq e^{(\varepsilon-2C)\vert V_n\vert} \leq e^{-C \vert V_n \vert}. \]
	Since $C>0$ was arbitrary, it follows that $\vol(\PP_n)= e^{-\omega\left(\vert V_n\vert\right)}$, as claimed.
\end{proof}

\section{Applications}\label{section: applications}
\subsection{Functions from hypercubes into $[0,1]$}
Let $Q_n=\{0,1\}^n$ denote the $n$-dimensional hypercube. Consider the sequence of sets $\left(Q_n\right)_{n \in \mathbb{N}}$ together with for every $N\leq n$ the collection of injections $\phi: \ Q_N\rightarrow Q_n$ obtained by choosing an arbitrary element $\mathbf{u}\in Q_n$ and an arbitrary set $S=\{s_1, s_2, \ldots s_N\}$
%\subseteq[n]$ of $N$ distinct integers 
of integers with $1\leq s_1<s_2<\ldots  < s_N\leq n$ and letting
\[\phi(\mathbf{v})_i =\left\{ \begin{array}{ll} 
v_j& \textrm{if } i=s_j\\
u_i & \textrm{otherwise}.
\end{array}\right. \]
It is an easy exercise to see that this constitutes a good, homogeneous ssee, which we denote by $\mathbf{Q}$. 
\begin{remark}
	In fact, we still get a good, homogeneous ssee if we consider any of the other natural collection of embeddings $\phi$ on hypercubes, such as for example all the embeddings $\phi$ obtained by selecting $\mathbf{u}\in Q_n$ and an injection $\psi:\ [N]\rightarrow [n]$ and letting
	\[\phi(\mathbf{v})_i =\left\{ \begin{array}{ll} 
	v_j+u_i \ (\mathrm{mod }  \ 2)& \textrm{if } i=\psi(j)\\
	u_i & \textrm{otherwise};
	\end{array}\right. \]
	but we do not pursue this here.
\end{remark}
Let $c\in \mathbb{R}_{>0}$. Recall that a function $f: \ Q_n \rightarrow [0,1]$ is \emph{$c$-Lipschitz} if changing a coordinate of $\mathbf{u}\in \{0,1\}^n$ changes the values of $f(\mathbf{u})$ by at most $c$.  Our aim in this subsection is to show that the probability a random function $f: \ Q_n\rightarrow [0,1]$ is $c$-Lipschitz is $c^{2^{n+o(1)}}$:
\begin{theorem}\label{thm: probability lipschitz}
	Let $c\in (0,1)$. Let $f: \ Q_n \rightarrow [0,1]$ be a random function	chosen according to the uniform measure on $[0,1]^{Q_n}$. Then
	\[\mathbb{P}\left(f \textrm{ is $c$-Lipschitz}\right)= c^{2^{n}+o(2^n)}.\]
\end{theorem}
We prove Theorem~\ref{thm: probability lipschitz} via a simple extremal entropy result, Theorem~\ref{theorem: lipschitz} below, from which Theorem~\ref{thm: probability lipschitz} can be easily deduced via Theorem~\ref{theorem: counting for homogeneous ssee}. Fix $c\in (0,1)$. Clearly, the collection of functions $Q_n\rightarrow[0,1]$ can be identified with the set of $[0,1]$-decorations of $Q_n$. Consider the hereditary property $\mathcal{P}$ of $[0,1]$-decorations of the (good, homogeneous) ssee $\mathbf{Q}$ corresponding to being $c$-Lipschitz.
\begin{theorem}\label{theorem: lipschitz}
	$\ex(Q_n, \mathcal{P}_n)=-\vert Q_n\vert \log c$.  
\end{theorem}
\begin{proof}
	For the upper bound, observe that the box $b=[0,c]^{Q_n}$ lies wholly inside $\mathcal{P}_n$ and has entropy exactly $-\vert Q_n\vert\log (c)$.

	For the lower bound, fix $\varepsilon>0$ with $c+\varepsilon<1$. Let $b\in \Boxes(Q_n)$ be such that $\mathrm{Ent}(b)= \ex(Q_n, \mathcal{P}_n)$ and $\vol\left(b\setminus \mathcal{P}_n\right)=0$. Then $b= \prod_{\mathbf{x}\in Q_n}A_{\mathbf{x}}$, where the $A_{\mathbf{x}}$ are measurable subsets of $[0,1]$. Set $\mathcal{X}=\{\mathbf{x}: \ \vert A_{\mathbf{x}}\vert> c\}$. Consider $\mathbf{x} \in \mathcal{X}$, and assume $\vert A_{\mathbf{x}}\vert=\ell$. Then for any $\eta>0$, there exists an interval $I=[u+\eta, u+\ell-\eta]\subseteq [0,1]$ such that both $A_{\mathbf{x}}\setminus [0,u+\ell-\eta]$ and $A_{\mathbf{x}}\setminus [u+\eta, 1]$ have strictly positive measure. 
%Since $\vol\left(b\setminus \mathcal{P}_n\right)=0$, this immediately implies that for every $\mathbf{y}\in Q_n$ obtained by modifying exactly one coordinate of $\mathbf{x}$, we have $\vert A_{\mathbf{y}}\cap [u+\ell-c-\eta, u+c-\eta]\vert =\vert A_{\mathbf{y}}\vert$. 
Let $\mathbf{y}\in Q_n$ be obtained by modifying exactly one coordinate of $\mathbf{x}$. Since $\vol\left(b\setminus \mathcal{P}_n\right)=0$ and $A_{\mathbf{x}}\setminus [0,u+\ell-\eta]$ has positive measure, the definition of $c$-Lipschitz implies that $A_{\mathbf{y}} \cap [0,u+\ell-\eta-c]$ must have zero measure. Similarly $A_{\mathbf{x}}\setminus [u+\eta, 1]$ having positive measure implies $A_{\mathbf{y}} \cap [u+\eta+c,1]$ has zero measure. Overall we obtain $\vert A_{\mathbf{y}}\cap [u+\ell-\eta-c, u+\eta+c]\vert =\vert A_{\mathbf{y}}\vert$. 
Since $\eta>0$ was arbitrarily chosen, this implies in fact that up to a zero-measure set $A_{\mathbf{y}}$  is contained inside the interval $[u+\ell-c, u+c]$. In particular we must have $\ell\leq 2c$, and $\vert A_{\mathbf{y}}\vert \leq (2c-\ell)<c$.

	Now  partition $Q_n$ into pairs $\{\mathbf{v}\times\{0\}, \mathbf{v}\times \{1\}\}$, with $\mathbf{v}$ running over all possible choices $\mathbf{v}\in Q_{n-1}$. By the above, we have that each such pair $\{\mathbf{x}, \mathbf{y}\}$ contains at most one element of $\mathcal{X}$. Moreover if this element is $\mathbf{x}$ and satisfies $\vert A_{\mathbf{x}}\vert= \ell$, then $\vert A_{\mathbf{y}}\vert\leq (2c-\ell)$, and 
	\[\vert A_{\mathbf{x}}\vert\cdot \vert A_{\mathbf{y}}\vert\leq  \ell\cdot (2c-\ell)<c^2.\]	
	Thus we have
	\begin{align*}
	\vol(b)=\prod_{\mathbf{v}\in Q_{n-1}}\vert A_{\mathbf{v}\times\{0\}}\vert \vert A_{\mathbf{v}\times\{1\}} \vert &\leq c^{2^n-2\vert \mathcal{X} \vert}c^{2\vert \mathcal{X} \vert}= c^{2^n},
	\end{align*}
	with equality if and only if $\mathcal{X}=\emptyset$ and for every $\mathbf{x}\in Q_n$ we have $\vert A_{\mathbf{x}}\vert=c$.	This shows that  $\ex(Q_n, \mathcal{P}_n)\geq  -\vert Q_n\vert\log c$, as claimed.
\end{proof}
\begin{proof}[Proof of Theorem~\ref{thm: probability lipschitz}]
	Let $\mathcal{P}$ denote, as above, the property of being $c$-Lipschitz, viewed as a hereditary property of $[0,1]$-decorations of the good homogeneous ssee $\mathbf{Q}$. By Theorem~\ref{theorem: lipschitz}, $\pi(\mathcal{P})= -\log c$, whence the volume estimate
	\[\vol(\mathcal{P}_n)= e^{\left(\log (c)+o(1)\right)\vert Q_n\vert }= c^{2^{n}+o(2^n)}\]
	follows immediately from Theorem~\ref{theorem: counting for homogeneous ssee}, implying the desired estimate for $\mathbb{P}\left(f \in \mathcal{P}_n\right)$.
\end{proof}
\subsection{Metric polytopes}\label{subsection: metric polytopes}
Given a set $S$ and a positive integer $s$, we write $S^{(s)}$ for the collection of subsets of $S$ of size $s$.
Let $K_n=(V,E)$ denote the complete graph on the vertex-set $V=V(K_n)=[n]$ with edge-set $E=E(K_n)=[n]^{(2)}$. It is an easy exercise to check that the sequence $\left(E(K_n)_{n \in \mathbb{N}}\right)$ together with the collection of embeddings $\phi$ corresponding to graph isomorphisms $K_N\rightarrow K_n$ constitutes a good homogeneous ssee, which we denote by  $\mathbf{K}$.

A $[0,1]$-decoration $d\in [0,1]^{E(K_n)}$ of the edges of $K_n$ can be seen as an assignment of distances to pairs of vertices of $K_n$.  Let $\mathcal{M}_n$ denote the collection of 
such $d$ which satisfy the triangle inequality (and for which $(V(K_n), d)$ is thus a metric space). The body $\mathcal{M}_n$ is known as the \emph{metric polytope}. The property $\mathcal{M}=\left(\mathcal{M}_n\right)_{n\in \mathbb{N}}$ is hereditary, since a subset of a metric space is itself a metric space. Thus we can use an easy extremal argument together with our main results to estimate the volume of $\mathcal{M}_n$, and thereby prove a weak form of a recent theorem of Kozma, Meyerovitch, Peled and Samotij~\cite{KozmaMeyerovitchPeledSamotij21}.

\begin{theorem}[Rough estimate for the volume of the metric polytope]\label{theorem: volume metric polytope}
$\vol(\mathcal{M}_n)=\left(\frac{1}{2}\right)^{\binom{n}{2}+o(n^2)}$.
\end{theorem}
\noindent Theorem~\ref{theorem: volume metric polytope} will follow from the following extremal result:
\begin{theorem}\label{theorem: extremal entropy metric polytope}
For all $n\geq 3$, $ \mathrm{ex}(\mathbf{K}, \mathcal{M})_n= \binom{n}{2}\log 2$.
\end{theorem}
\begin{proof}
For the upper bound, observe that the box $b_n=[\frac{1}{2},1]^{E(K_n)}\subseteq \mathcal{M}_n$ and has  entropy	$\binom{n}{2}\log 2$ (since clearly for every $d\in b_n$, the associated assignment of distances to the edges of $K_n$ satisfies the triangle inequality). Thus $ \mathrm{ex}(\mathbf{K}, \mathcal{M})_n\leq \binom{n}{2}\log 2$ for all $n\geq 3$.

For the lower bound, consider a simple box $b\in \mathrm{Box}(E(K_3))$ with $\vol(b\setminus \mathcal{M}_3)=0$. Then $b=I_1\times I_2\times I_3$, where $I_i$ is a union of nonempty intervals (and corresponds to the edge $[3]\setminus\{i\}$ of $K_3$). Set $a_i=\min I_i$ and $b_i=\max I_i$. Clearly, we have $b\subseteq \prod_i[a_i,b_i]\subseteq \prod_i[a_i,1]$. In particular we have
\begin{align}\label{eq: volume bound triangle polytope}
\vol(b)\leq \prod_{i} (b_i-a_i) \quad \text{and} \quad \vol(b)\leq \prod_i (1-a_i).
\end{align}
Let $A=\sum_i a_i$. Since $\vol(b\setminus \mathcal{M}_3)=0$, it follows by the triangle inequality that for all $i\in [3]$, $b_i\leq a_{i+1}+a_{i+2}$ (where the indices are taken modulo 3). Summing over all $i$ and subtracting $A$ from both sides, we get
$\sum_i (b_i-a_i)\leq A$.
Simple calculus then tells us that for $A$ fixed, we have $\prod_{i} (b_i-a_i)\leq \left(A/3\right)^3$. On the other hand, again by simple calculus, for $A$ fixed we have $\prod_{i} (1-a_i)\leq \left(1-A/3\right)^3$. Combining these two bounds with~\eqref{eq: volume bound triangle polytope}, we get that 
\begin{align*}
\vol(b)\leq\min\Bigl(\left(A/3\right)^3, \left(1-A/3\right)^3\Bigr)=2^{-3}.
\end{align*}
As $b$ was an arbitrary simple box with 	$\vol(b\setminus \mathcal{M}_3)=0$, this implies
$\ex(\mathbf{K}, \mathcal{M})_3=\binom{3}{2}\log 2$. Since $\mathbf{K}$ is homogeneous, it follows from Proposition~\ref{proposition: existence of limits from uniformity} that $\ex(\mathbf{K}, \mathcal{M})_n\geq \binom{n}{2}\log 2$ for all $n\geq 3$, as required.
\end{proof}
\begin{proof}[Proof of Theorem~\ref{theorem: volume metric polytope}]
	Immediate from Theorems~\ref{theorem: extremal entropy metric polytope} and~\ref{theorem: counting for homogeneous ssee} applied to the good homogeneous ssee $\mathbf{K}$.
\end{proof}
The problem of estimating $\vol(\mathcal{M}_n)$ has been previously considered by several other researchers, who obtained significantly stronger estimates than Theorem~\ref{theorem: volume metric polytope}. As observed by Kozma, Meyerovitch, Peled and Samotij, the problem of estimating $\vol(\mathcal{M}_n)$ is related to the problem of estimating the number of metric spaces  on $n$ points with integer distances, which was studied by Mubayi and Terry~\cite{MubayiTerry19b} using the container method. Balogh and Wagner~\cite[Theorem 3.7]{BaloghWagner16} used the container method to show
$\vol(\mathcal{M}_n) \leq \left(\frac{1}{2}\right)^{\binom{n}{2}+n^{11/6+o(1)}}$. Finally, in a recent and impressive paper using entropy techniques, Kozma, Meyerovitch, Peled and Samotij obtained much more precise estimates on $\vol(\mathcal{M}_n)$: they proved in ~\cite[Theorem 1.2]{KozmaMeyerovitchPeledSamotij21} that
\[  \left(\frac{1}{2}\right)^{\binom{n}{2}} e^{\frac{n^{3/2}}{6}+o(n^{3/2})}\leq \vol(\mathcal{M}_n) \leq \left(\frac{1}{2}\right)^{\binom{n}{2}} e^{O(n^{3/2})}.\]
They also showed in Section~5.3 of the paper, in joint work with Morris, how using the more precise container theorems of~\cite{BaloghMorrisSamotij15} (rather than the simple containers of~\cite{SaxtonThomason16} that underpin Theorem~\ref{theorem: ssee container} and all the results in this paper) could be made to yield slightly weaker upper bounds of $\left(\frac{1}{2}\right)^{\binom{n}{2}} e^{O(n^{3/2}(\log n))}$ on $ \vol(\mathcal{M}_n)$, improving on the earlier container results of Balogh and Wagner. The point of Theorem~\ref{theorem: volume metric polytope} above is thus not to prove a new or optimal upper bound, but rather to illustrate how the results of this paper give a simple, streamlined approach to such problems.

\subsection{Weighted graphs}
A $[0,1]$-decoration $w$ of the edges of $K_n$ may be identified may be viewed as an assignment of weights $w(e)\in [0,1]$ to the edges $e\in E(K_n)$. Given such a decoration, we may define the weight of a set $S\subseteq V(K_n)$ as $w(S):=\sum_{e\in S^{(2)}}w(e)$.  For a fixed integer $s\geq 2$ and a real number $r\in (0, \binom{s}{r}]$, let $\PP(s,r)$ be the hereditary property of $[0,1]$-decorations of the edges of  $\mathbf{K}$ corresponding to having no $s$-set of vertices whose weight exceeds $r$.
 \begin{theorem}\label{theorem: weighted graphs}
For all $n\geq s$ we have $\ex(E(K_n), \PP_n(s,r))=\binom{n}{2}\log \Bigl(\frac{\binom{s}{2}}{r}\Bigr)$, with equality uniquely attained up to zero-measure sets by the box $\Bigl[0, \frac{r}{\binom{s}{2}}\Bigr]^{E(K_n)}$.
\end{theorem}
\begin{proof}
For the upper bound, note that $\Bigl[0, \frac{r}{\binom{s}{2}}\Bigr]^{E(K_n)}$ is a box lying wholly inside $\mathcal{P}_n$ and having the claimed volume. For the lower bound, note first of all that up to a zero-measure set, any entropy-maximiser $b$ for $\mathcal{P}_n$ must be of the form $b=\prod_{e\in E(K_n)}[0, w(e)]$, where $w: \ E(K_n)\rightarrow [0,1]$ is an edge-weighting from $\mathcal{P}_n$. By averaging the weight $w(S)$ over all $s$-sets $S$ and using the fact that $w\in \mathcal{P}(s,r)$, we see that
\begin{align*}
\binom{n-2}{s-2}w(V(K_n))=\sum_{S\in [n]^{(s)}}w(S)\leq r\binom{n}{s}.
\end{align*}
In particular, $\sum_{e\in E(K_n)} w(e)= w(K_n)\leq \frac{r}{\binom{s}{2}} \binom{n}{2}$. It then follows from the AM-GM inequality that
\begin{align*}
\vol(b) =\prod_{e\in E(K_n)} w(e)\leq \left(\frac{r}{\binom{s}{2}}\right)^{\binom{n}{2}},
\end{align*}
which gives the required lower bound on the entropy of $b$. Furthermore the AM-GM inequality also implies that, equality is attained if and only if $w(e)=[0, r/\binom{s}{2}]$ for every $e\in E(K_n)$, i.e. if and only if up to a zero-measure set $b$ is equal to the box $[0,w]^{E(K_n)}$, as claimed.
\end{proof}

\begin{corollary}\label{cor: counting weighted graphs}
$\vol(\PP_n(s,r))= \Bigl(\frac{r}{\binom{s}{2}}+o(1)\Bigr)^{\binom{n}{2}}$.	
\end{corollary}
\begin{proof}
Theorem~\ref{theorem: weighted graphs} shows $\pi(\PP(s,r))= r/\binom{s}{2}$. The result is then immediate from an application of Theorem~\ref{theorem: counting for homogeneous ssee} to the good, homogeneous ssee $\mathbf{K}$.
\end{proof}

We should note here that Mubayi and Terry~\cite{MubayiTerry19, MubayiTerry20}  considered the related problem of maximising the product of edge-multiplicities in multigraphs in which every $s$-set of vertices spans at most $r$ edges. In this case the fact that edge-multiplicities have to be positive integers completely changes the nature of the problem, which  becomes highly nontrivial (see~\cite{DayFalgasRavryTreglown20, FalgasRavry21} for recent progress on the Mubayi--Terry problem).

\section{Entropy of $[k]$-decorated graph limits}\label{section: graph limits}
As another application of our work, we prove a generalisation of the result of Hatami, Janson and Szegedy on the entropy of graph limits to $[k]$-decorated graph limits. Hatami, Janson and Szegedy defined and studied the entropy of a graphon in~\cite{HatamiJansonSzegedy18} .  They used this notion to give an alternative proof of the Alekseev--Bollob\'as--Thomason Theorem~\cite{Alekseev93, BollobasThomason97} and to describe the typical structure of a graph in a hereditary property. The Hatami--Janson--Szegedy notion of entropy can be viewed as a graphon analogue of the classical notion of the entropy of a discrete random variable. Generalising their result to $[k]$-decorated graphons was one of the original motivations for our foray into container theory (as containers allow for an easy transfer of certain results to the limit setting). In fact, we sought unsuccessfully to obtain a generalisation to $[0,1]$-decorated graph limits, which we now define.

 Let $\mathcal{K}$ be a compact second-countable Hausdorff space. A $\mathcal{K}$-decorated graph is a function $w: E(K_n) \to\mathcal{K}$ assigning to each edge of the complete graph $K_n$ a label from $\mathcal{K}$.  In~\cite{LovaszSzegedy10}, Lov\'asz and Szegedy initiated the study of the limits of sequences of $\mathcal{K}$-decorated graphs. Generalising the well established theory of graph limits, given a $\mathcal{K}$-decorated graph $G$ they defined homorphism densities $t(F, G)$ for $C[\mathcal{K}]$-decorated graphs $F$ in $G$, where $C[\mathcal{K}]$ is the collection of all continuous functions $\mathcal{K}\rightarrow \mathbb{R}$. They defined convergence relative to this notion of homomorphism density, and showed the limit objects in this theory were \emph{$\mathcal{K}$-graphons}, which are symmetric measurable functions $W: \ [0,1]^2\rightarrow \mathcal{M}(\mathcal{K})$, where $\mathcal{M}[\mathcal{K}]$ denotes the set of Borel probability measures on $\mathcal{K}$.  Further contributions to the study of $\mathcal{K}$-decorated graph limits were  made by Kunszenti-Kov{\'a}cs, Lov{\'a}sz and Szegedy~\cite{KKLS14}, who defined a modified notion of cut distance (called jumble distance, which they used to provide weak regularity lemma and a counting lemma for $\mathcal{K}$-graphons) and showed that the space of $\mathcal{K}$-decorated graph limits was closed under the homomorphism density notion of convergence. However a number of questions, such as the uniqueness of the representation of the limits of  $\mathcal{K}$-decorated graph sequences or the compactness of that space under their modified notion of cut distance, remain open.

Our goal was to extend the results of Hatami--Janson--Szegedy on graphon entropy to $[0,1]$-decorated graphons. However, we were unable to show that every sequence of $[0,1]$-decorated graphs contains a convergent subsequence of $[0,1]$-decorated graphs. Thus we had to content ourselves with proving a generalisation of Hatami--Janson--Szegedy  to the easier setting of $[k]$-decorated graphs, which we now give. Before we can state our results, we must recall the definitions of templates and realisations from Section~\ref{section: proofs}, and also recall from~\cite{FalgasRavryOconnellUzzell18} that the entropy of a $[k]$-colouring template $t$ for $K_n$ is 
\[ \Ent(t) := \log_k \prod_{e \in E(K_n)} |t(e)|. \]
Note that for any template $t$ we have $0 \leq \Ent(t) \leq \binom{n}{2}$, and the number of realisations of $t$ is exactly $| \langle t \rangle | = k^{\Ent(t)}$. 
For a hereditary property $\mathcal{P}$ of $[k]$-decorated graphs, we defined
\[ \ex(n,\mathcal{P}_n) := \max \{ \Ent(t): \text{ $t$ is a $k$-colouring template for $K_n$ with $\langle t \rangle \subseteq \mathcal{P}_n$} \}. \]
Finally, let us recall the $[k]$-decorated graph analogue of Theorem~\ref{theorem: counting for homogeneous ssee} from~\cite{FalgasRavryOconnellUzzell18}.
\begin{theorem}[Corollary 2.15 in \cite{FalgasRavryOconnellUzzell18}]\label{theorem: Cor2.15 from companion}
Let $\mathcal{P}$ be a hereditary property of $[k]$-decorated graphs with $\PP_n \not= \emptyset$ for every $n \in \mathbb{N}$ and let $\eps >0$ be fixed. There exists $n_0 \in \mathbb{N}$ such that for all $n \geq n_0$, 
\[ k^{\pi(\PP) \binom{n}{2}} \leq |\PP_n | \leq k^{(\pi(\PP)+\eps)\binom{n}{2}}.\]
\end{theorem}
\noindent Given a $[k]$-graphon $W$ and $i\in [k]$, we denote by $W_i (x,y) :=W(x,y)(i)$ the probability of $\{i\}$ under the probability measure $W(x,y)$ on $[k]$. As noted by Lov\'asz and Szegedy~\cite[Example 2.8]{LovaszSzegedy10} each $W_i$ is a graphon. Given a probability measure $P$ on $[k]$ with $P(i)=p_i$, we define the \emph{$k$-ary entropy} of $P$ to be
\[ h_k(P) := \sum_{i \in [k]} -p_i \log_k p_i.\]
Then the \emph{entropy} of a $[k]$-graphon $W$ is 
\[ \Ent(W) := \int \int_{[0,1]^2} h_k (W_k(x,y)) dA.\]
Note that $0 \leq \Ent(W) \leq 1$. For $k = 2$ our definition of decorated graphon entropy coincides with that of Hatami, Janson and Szegedy. Given a property $\PP$ of $[k]$-decorated graphs, we denote by $\hat{\PP}$ the closure under the cut norm (see Section~\ref{subsection: cut distance} below for a definition of the cut norm) of the collection of $[k]$-graphons that can be obtained as a limit of a convergent sequence of elements of $\PP$.  We can at last state the main result of this section. 
\begin{theorem}\label{theorem: maximum entropy gives growth rate for [k]}
Let $\PP$ be a hereditary property of $[k]$-decorated graphs and let $m(\hat{\PP}):=\max_{W\in \hat{\PP}} \Ent(W)$.
Then 
\[ \lim_{n \to \infty} \frac{\log_k |\PP_n|}{\binom{n}{2}} = m(\hat{\PP}).\]
\end{theorem}
Given Theorem~\ref{theorem: Cor2.15 from companion}, the theorem above is equivalent to the assertion that $\pi(\PP)=m(\hat{\PP})$, which in fact is what we shall prove.
%The short proof will follow quickly once we introduce some results from graph limit theory.

\subsection{A cut distance for $[k]$-graphons}\label{subsection: cut distance}
We require convergence in our proof, but rather than using convergence with respect to homomorphisms, we use convergence with respect to an appropriately defined cut distance. Frieze and Kannan~\cite{FriezeKannan99} introduced a cut norm $\cutnorm{\cdot}$ that has become central to the theory of graph limits. (see~\cite[Section~4]{Janson13} for an overview of the history of the cut norm in other contexts).  The \emph{cut norm} of a graphon~$W$ is
\begin{align*}%\label{eq:CutNormDef}
\cutnorm{W} := \sup_{S, T \subseteq \oi} \biggl\lvert \int_{S \times T} W(x, y)\,dx dy\biggr\rvert,
\end{align*}
where the supremum is over all pairs $(S,T)$ of measurable subsets of~$\oi$.  If $U$ and $W$ are graphons, then
\begin{align*}%\label{eq:CutDistanceDef}
\dcut(U, W) := \cutnorm{U - W} = \sup_{S, T \subseteq \oi} \biggl\lvert \int_{S \times T} \bigl(U(x, y) - W(x, y)\bigr)\,dx dy\biggr\rvert.
\end{align*}
Given a measure-preserving transformation~$\varphi : \oi \to \oi$, we define $W^{\varphi}$ by $W^{\varphi}(x, y) := W(\varphi(x), \varphi(y))$.  The \emph{cut distance} between $U$ and $W$ is
\begin{align*} %\label{eq:deltakcutDef}
\deltacut(U, W) := \inf_{\varphi : \oi \to \oi} \dcut(U, W^{\varphi}),
\end{align*}
where the infimum is taken over all measure-preserving transformations~$\varphi : \oi \to \oi$. 
%Lov\'asz and Szegedy~\cite{LovaszSzegedy10} did not consider a version of the cut distance for $[k]$-decorated graphs.
We introduce an appropriate generalisation of the cut distance here (this was also previously considered in~\cite{KKLS14}). If $G$ and $H$ are two $[k]$-decorated graphs with edge labellings $g,h: E(K_n) \to [k]$ respectively and with vertex set $[n]$, we define
\[ \dkcut(G,H) := \max_{S,T \subseteq [n]} \frac{1}{n^2} \sum_{i=1}^{k} \left| \sum_{(u,v) \in S \times T} (\mathbbm{1}(g(uv)=i)-\mathbbm{1}(h(uv)=i)) \right|.\]
If $U$ and $W$ are $[k]$-graphons, we define
\[ \dkcut(U,W) := \sup_{S,T \subseteq [0,1]} \sum_{i=1}^{k} \left| \int_{S \times T} (U_i(x,y) - W_i(x,y) dxdy \right|.\]
We define the cut distance $\deltakcut$ for $[k]$-graphons analogously to the definition for graphons, mutatis mutandis. Letting $\mathcal{W}_k$ denote the set of all $[k]$-graphons, we let $\widetilde{\mathcal{W}_k}$ denote the quotient of $\mathcal{W}_k$ obtained by identifying $U$ and $W$ whenever $\deltakcut(U,W)=0$. 
\begin{theorem}\label{theorem: graphon compactness}
The space $(\widetilde{\mathcal{W}_k}, \deltakcut)$ is compact.
\end{theorem}
The proof is essentially identical to the proof of compactness with respect to cut distance for ordinary graphons as given by the original argument of Lov\'asz and Szegedy~\cite[Theorem 5.1]{LS:analyst}. Therefore we only give a brief sketch here. 
\begin{proof}[Sketch proof of Theorem~\ref{theorem: graphon compactness}]
First note that a weak regularity lemma for $[k]$-graphons follows very quickly from the weak regularity lemma for ordinary graphons (Lemma 3.1 in~\cite{LS:analyst}), simply by running it for each $W_i(x,y)$, $i \in [k]$ simultaneously.
Now given a sequence of $[k]$-graphons $W_1,W_2,\ldots \in \widetilde{\mathcal{W}_k}$, using the weak regularity lemma, we find for every $n\in \mathbb{N}$ a sequence of step-functions $W_{n ,\ell}$, $\ell \in \mathbb{N}$, converging to $W_n$. Now for each $\ell$ we find a subsequence $n_i$, $i \in \mathbb{N}$ for which $W_{n_i,\ell}$ converges in cut distance to a $[k]$-graphon $U_\ell$. Then by the Martingale Convergence Theorem, the sequence $(U_\ell)$ converges to a limit $U$. Finally one can show $\deltakcut(W_i,U) \to 0$ using a $3\eps$-argument.
\end{proof}

\subsection{Going between templates and graphons}
Fundamental to the theory of graph limits is a natural way of obtaining a graphon from a given graph, 
and conversely (via sampling) a way of obtaining a graph on $n$ vertices from a given graphon.
These transformations respect homomorphism densities and cut distance, and in particular, with probability one,
a sequence of $n$-vertex graphs $(G_n)_{n \in \mathbb{N}}$ sampled from $W$ converges to $W$ itself (as established in~\cite{Borgsetal08}).
Similarly here, we obtain a way of going between the discrete and limit objects.
The only property we require is that this transformation respects entropy, which as we will see follows easily from the definitions. 

Given a set of $n$ points $x_1,\dots,x_n$ from $[0,1]$ and a $[k]$-graphon $W$, 
we may define a $k$-colouring template for $K_n$, $t_W[x_1,\dots,x_n]$, by setting
$t(ij):=\{ c\in [k]: \mathbb{P}(W(x_i, x_j)=c) >0\}.$ 
Further we may define a random $k$-colouring $c_W[x_1,\dots,x_n]$ by setting
$c(ij)$ to be a random colour from $[k]$ drawn according to the probability distribution given by $W(x_i, x_j)$. We define the $W$-random template $t_W(n)$ and the $W$-random colouring $c_W(n)$ by selecting $x_1,\dots,x_n$ uniformly at random from $[0,1]$ 
and then taking the resulting (induced) $k$-colouring template and random $k$-colouring respectively.

Our $W$-random templates and colourings give us a way of going from $[k]$-graphons to $k$-colouring templates and $k$-colourings of $E(K_n)$. 
We can also go in the other direction: first divide $[0,1)$ into intervals $I_i:=[(i-1)/n,i/n)$ for $1 \leq i \leq n$. 
Given a $k$-colouring template $t$ of $K_n$, we may define $W_t$ by defining for every $(x,y) \in [0,1]^2$ and $c \in [k]$,
\[ \left(W_t\right)_c := \begin{cases} 
\frac{1}{|t(ij)|} \mathbbm{1}(c \in t(ij)) & \text{ if } (x,y) \in I_i \times I_j, 1 \leq i,j \leq n,  \ i\neq j; \\
\frac{1}{k} & \text{ if } (x,y) \in I_i \times I_i, 1 \leq i \leq n.\\
\end{cases}\]
In other words, for each tile $I_i \times I_j$ we distribute the mass evenly over the colours which appear in $t(ij)$, and give the uniform distribution to the diagonal tiles $I_i \times I_i$.
 
By viewing a $k$-colouring $c$ of $E(K_n)$ as a (zero entropy) template, we may in the same way obtain from it a $[k]$-graphon $W_c$. 
Thus we may go in a natural way from properties of colourings to properties of decorated graphons, and vice-versa.

Note that for all $k$ and $n$, and every $k$-colouring template $t$ for $K_n$, we have $\binom{n}{2} \Ent(W_t)=\Ent(t)+\frac{n-1}{2}$. In particular for a template $t$ and its associated $[k]$-decorated graphon $W_t$ we have
\begin{align}\label{equation: entconversion}
\Ent(W_t) = \frac{\Ent(t) + O(n)}{\binom{n}{2}}.
\end{align}
Furthermore in the reverse direction, given a $[k]$-decorated graphon $W$, $\binom{n}{2}\Ent(W)-\frac{n-1}{2}$ is exactly the expected value of the discrete $k$-ary entropy of the $W$-random colouring model $t_W(n)$.

\subsection{Proof of main result}
\begin{proof}[Proof of Theorem~\ref{theorem: maximum entropy gives growth rate for [k]}]
For each $n \in \mathbb{N}$ take an extremal template $t_n$ which maximises $\ex(n,\PP_n)$.
We have
\[ \pi(P)= \lim_{n \to \infty} \frac{ \ex(n, \mathcal{P}_n)}{\binom{n}{2}} = \lim_{n \to \infty} \frac{ \Ent(t_n)}{\binom{n}{2}}.\] 
Letting $W_{t_n}$ be the $[k]$-graphon corresponding to $t_n$, 
we have $\Ent(W_{t_n})=(\Ent(t_n)+O(n))/ \binom{n}{2}$ by~(\ref{equation: entconversion}). 
By Theorem~\ref{theorem: graphon compactness}, there exists a subsequence $(W^{j})_{j \in \mathbb{N}}$ of $\left(W_{t_n}\right)_{n\in \mathbb{N}}$ which converges to a limit $[k]$-graphon $W$, where we have $\deltakcut(W^j,W) \to 0$ as $j \to \infty$.  Now since $\langle t_n\rangle \subseteq \mathcal{P}_n$, for any $N$ fixed, the probability that the $W_{t_n}$-random colouring of $K_N$ is in $\mathcal{P}_N$ is $1-O(N/n)=1-o(1)$. It follows that a $W$-random colouring lies in $\mathcal{P}$ with probability $1$, and thus $W\in \hat{\PP}$. Since entropy is a linear functional we thus have $m(\hat{\PP}) \geq \Ent(W) =\pi(P)+o(1)$.

Conversely, let $W$ be an entropy maximiser in $\hat{\PP}$ 	with $\Ent(W)=m(\hat{\PP})$. 
For every $n \in \mathbb{N}$, by linearity of expectation if $n$ points $x_1,\dots,x_n$ are chosen uniformly at random from $[0,1]$ then with strictly positive probability
\[ \Ent(t_W[x_1,\dots,x_n]) \geq m(\hat{\PP}) \binom{n}{2}. \]
Furthermore, as $W \in \hat{\PP}$, almost surely $\langle t_W[x_1,\dots,x_n] \rangle \subseteq \PP_n$,
which implies that $|\PP_n| \geq k^{m(\hat{\PP}) \binom{n}{2}}$ for every $n \in \mathbb{N}$.
By Theorem~\ref{theorem: Cor2.15 from companion}, 
we have $|\PP_n| \leq k^{(\pi(\PP)+o(1))\binom{n}{2}}$,
and thus $\pi(\PP)+o(1) \geq m(\hat{\PP})$ as required.
\end{proof}

\section{Concluding remarks}\label{section: conclusion}
In this paper, we have explored some consequences of the simple versions of the container theorems of Balogh--Morris--Samotij and Saxton--Thomason for the problem of estimating volume or approximating by boxes for certain hereditary bodies. Many problems remain open however.

\subsection{Alternative approaches to containers?}
Using the Saxton--Thomason simple container theorem as a black box (which, as we stated at the beginning of Section~\ref{section: proofs}, is the result behind Theorem~\ref{theorem: ssee container} and thus the main tool behind all our results), we showed in Theorem~\ref{theorem: containers for hereditary bodies} that hereditary properties of $[0,1]$-decorated ssee-s can be approximated by a `small' union of boxes.

A natural question to ask is whether one can go in the other direction: is it possible to obtain a container theorem from purely geometric considerations on approximations of hereditary bodies by boxes? A simplest version of this question is the following: suppose we have a sequence $(b_n)_{n\in \mathbb{N}}$ of bodies with $b_n\subseteq [0,1]^n$ and every strict projection of $b_n$ into $[0,1]^N$ (where we use strict projection in the sense of Definition~\ref{definition: projections}) is a subset of $b_N$. Does it follow (by measure-theoretic/geometric arguments) that for all $n$ sufficiently large there exists a `fine' approximation of $b_n$
by a `small' collection of simple boxes?

\subsection{Questions about graph limits}
In a different direction, we have tried to connect some container-derived results with questions about limit objects. A natural question is, again, whether one can go in the other direction, and derive some finitary container theorems from infinitary arguments about limit objects?

The simplest example of this is perhaps the following: suppose we have a hereditary property $\PP$ of $\{0,1\}$-decorations of $E(K_n)$ (i.e. of ordinary graphs). Let $\hat{\PP}$ denote the closure of the family of limits of sequences of graphs from $\PP$ under the cut norm. Let $\mathcal{Q}$ denote the collection of graphons that lie at graph distance at most $\varepsilon$ from $\hat{\PP}$ --- this is a closed and hence compact set. Introduce a partial order on $\mathcal{Q}$ by setting $W_1\succ W_2$ if almost everywhere either $\mathrm{Ent}(W_1(x,y))>0$ or $W_1(x,y)=W_2(x,y)$ holds. Then for each $W\in \mathcal{Q}$, let $B(W)$ denote the interior of the collection of $W'\in \PP$ with $W\succ W'$. Clearly the $B(W)$ are open sets in the closed, compact set $\PP$. Thus if one could show that they also cover $\PP$ it would follow by compactness that there exists some finite set $S$ (with size depending on $\varepsilon$) of elements of $\mathcal{Q}$ such that $\bigcup_{W\in S}B(W)=\hat{\PP}$. One could then plausibly extract from the graphons in $S$ a small family of containers for $\PP_n$. This or other approaches to the construction of containers `from the limit' and from purely analytic considerations strike us as an intriguing problem.

With regards to limit objects, the other obvious question is generalising Theorem~\ref{theorem: maximum entropy gives growth rate for [k]} to $[0,1]$-decorated graphons. Here our problem is that we did not prove compactness of the limit space under the cut distance (i.e. we do not have a $[0,1]$-decorated version of Theorem~\ref{theorem: graphon compactness}), and so given a sequence of boxes from $\PP_n$ we could not extract a subsequence converging to an element of $\hat{\PP}$, which we needed to bound $m(\hat{P})$ below. Addressing this issue would immediately extend our results to $[0,1]$-decorated graphons and in addition would advance the project of Lov\'asz and Szegedy of building a theory for $\mathcal{K}$-decorated graph limits for second-countable compact Hausdorff spaces $\mathcal{K}$, a worthwhile goal in itself.

\subsection{Quality of the container approximation}
Can one improve assumption (ii) in Theorem~\ref{theorem: containers for Forb(finite union of boxes)} (and hence Theorem~\ref{theorem: containers for hereditary bodies})? For instance, could we guarantee that, say 
	\[\vol\biggl(\bigcup_{c\in \containers}c\biggr)\leq C_{\varepsilon}\vol(\PP_n)\]
	for some $n$-independent constant $C_{\varepsilon}>1$?
	Or could one show a weaker bound of the form 
	\[\vol(c\setminus \PP_n)< \varepsilon \vol(c)?\]
 	Putting it in slightly different terms: how fine can we make our approximation of a hereditary body $b$ by simple boxes? There should be a trade-off between the fineness of our approximation and the number of boxes it contains. Is it the case that e.g.~the worst-case product of the approximation ratio and the size of the approximation family is bounded below by some function of~$\vert V_n \vert$? Further, what do the bodies that are hardest to approximate look like?
	
\subsection{Relaxing homogeneity}	
	In Theorem~\ref{theorem: counting for homogeneous ssee} we obtained a rather clean statement concerning the volume of hereditary properties for homogeneous ssee-s. Homogeneity is a strong condition, however, and it natural to ask whether it can be relaxed. Explicitly, call a ssee $\mathbf{V}$ \emph{almost homogeneous} if there exist constants $C>c>0$ such that for every $n\geq N$, every $x \in V_n$ is contained in at least $c\bigl\vert\binom{V_n}{V_N}\bigr\vert\frac{\vert V_N\vert }{\vert V_n\vert }$ and at most $C\bigl\vert\binom{V_n}{V_N}\bigr\vert\frac{\vert V_N\vert }{\vert V_n\vert }$  embeddings $\phi(V_N)$ with $\phi \in \binom{V_n}{V_N}$.

	Can one obtain a version of Theorem~\ref{theorem: counting for homogeneous ssee} in which the homogeneity assumption is relaxed to almost homogeneity? 
	 This would increase the generality of the results in this paper and allow us to cover some important cases, such as that of the ssee $\mathfrak{I}$ where $\mathfrak{I}_n=[n]$ and the embeddings $\phi: \mathfrak{I} _N\rightarrow \mathfrak{I}_n$ consist of the injections from $[N]$ into arithmetic progressions of length $N$ in $[n]$.  Another example would be that of permutations, see Section~\ref{subsection: sparse bodies} below.
	
	\subsection{Containers for thin bodies}\label{subsection: sparse bodies}
	%		\[\frac{\vol\left(\bigcup_{c\in \containers_n}\right)}{\vol(\PP_n)} \vert \containers_n\vert \geq  \]
	In this paper, we have been content with a simple container bound $\vert \mathcal{C}\vert \leq e^{\varepsilon \vert V_n\vert}$ on the size of the container family $\mathcal{C}$.  This is sufficient to estimate the volume of $\PP_n$ when the maximum volume of a box contained in $\PP_n$ (up to a zero-measure set) is of order $e^{-\theta\left(\vert V_n\vert \right)}$. However, for `thinner' bodies when this extremal volume is of order $e^{-\omega\left(\vert V_n\vert \right)}$, our results (more specifically Theorem~\ref{theorem: counting for homogeneous ssee}) say nothing more precise  than $\mathrm{vol}(\PP_n)= e^{-\omega(\vert V_n\vert )}$. This is definitely a limitation of our work --- the original container theorems of Balogh--Morris--Samotij and Saxton--Thomason  can give much better estimates, but require information on the degree measure (something which, as Saxton and Thomason~\cite{SaxtonThomason16} observe is unnecessary in the case of `thicker' bodies).

	There are a number of interesting examples within our framework where more precise estimates would be advantageous --- for instance, that of permutations, which we discuss below.

	Denote by $S_n$ the collection of all permutations of $[n]$.  Given a $[0,1]$-decoration of $[n]$, we may define a permutation in $S_n$ as follows. Let $\mathcal{B}$ denote the collection of $x\in [0,1]^{[n]}$ which have at least~two coordinates equal. Note that this is a zero-measure set. Given $x\in [0,1]^{[n]}\setminus \mathcal{B}$, we equip $[n]$ with a linear order $\leq_x$ by setting $i\leq_x j$ if $x_i \leq x_j$.  Then for each $i$, let $\sigma_x(i)$ denote the rank of $i$ in this order. Clearly $\sigma_x$ is a permutation of $[n]$, and every permutation can be realised in this way. Conversely, to each permutation $\sigma$ of $[n]$ we may associate the body $b_{\sigma}$ of all $x\in [0,1]^{[n]}\setminus \mathcal{B}$ such that $\sigma_x=\sigma$. Observe that (i) if $\sigma$, $\sigma'$ are distinct elements of $S_n$, then $b_{\sigma}$ and $b_{\sigma'}$ are disjoint subsets of $[0,1]^{[n]}$, and (ii) $\vol(b_{\sigma})=1/n!$ for all $\sigma\in S_n$. (It is also worth remarking that given a body $b\subseteq [0,1]^{V_n}$, we can define a $b$-random permutation by selecting $x\in b$ uniformly at random. This gives an interesting non-uniform model for random permutations.)

	Given permutations $\sigma \in S_N$ and $\tau \in S_n$, we say $\sigma$ is a \emph{subpattern} of $\tau$ if there is an order-preserving injection $\phi: \ [N]\rightarrow [n]$ such that for every $i,j \in [N]$, $\sigma(i)< \sigma(j)$ if and only if $\tau(\phi(i))<\tau(\phi(j))$. One important topic of study in permutation theory is that of \emph{pattern avoidance}. Can one count or characterise the permutations in $S_n$ avoiding a given pattern $\sigma\in S_N$?
	
	\begin{definition}
		A \emph{permutation class} $\PP$ is a sequence $(\PP_n)_{n\in \N}$ of subsets $\PP_n\subseteq S_n$ which is closed under taking subpatterns. (I.e.~if $\tau \in \PP_n$ and $\sigma\in S_N$ is a subpattern of $\tau$ then $\sigma \in \PP_N$.)
	\end{definition}
	\begin{definition}
		Given a permutation $\pi \in S_N$, let $S_n(\pi)$ denote the collection of all $\tau \in S_n$ that do not contain $\pi$ as a subpattern.

		\noindent The \emph{Stanley-Wilf limit} of the permutation class $\mathcal{S}= (S_n(\pi))_{n \in \N}$ is
		\[L(\pi):= \lim_{n\rightarrow \infty} \vert S_n(\pi)\vert^{1/n}. \]
	\end{definition}\noindent (The existence of the limit~$L(\pi)$ is highly non-trivial.  Its existence was established by Marcus and Tardos in 2004 in~\cite{MarcusTardos04}.) 	We observe here that pattern-avoidance and Stanley-Wilf limits for permutation classes fits very nicely within the framework of $[0,1]$-decorated ssee-s.  To wit: let $\V$ be the ssee with $V_n=[n]$ and $\binom{V_n}{V_N}$ being the collection of all order preserving injections $\phi: \ V_N\rightarrow V_n$. One can easily check that this constitutes a good ssee. Given a forbidden pattern~$\pi \in S_N$, let $\PP = \Forb(b_{\pi})$. Clearly, we have $\vert S_n(\pi)\vert^{1/n} = \left(n! \vol (\PP_n)\right)^{1/n}$.  Thus providing a good estimate on $\vol(\PP_n)$ via containers could potentially give a good estimate on $L(\pi)$. However, as pointed out in the previous subsection, the consequences of simple container theory obtained in this paper are not sufficiently precise to do so: in the language of graph theory, what we study corresponds to the `dense' case with strictly positive Tur\'an density, whereas pattern avoidance belongs to the `sparse' case with zero Tur\'an density.

	The example of permutations suggests it would be interesting to obtain versions of Theorem~\ref{theorem: containers for hereditary bodies} that work in a sparser setting, i.e.~with sharper estimates on the size of the container family $\containers$ than are given by (iii). In this case, one will have to go back to the original theorems of Balogh--Morris--Samotij and Saxton--Thomason, rather than use the simple (but weaker) container theorem of Saxton--Thomason as a black box.

%\textcolor{red}{Applications to the study of sharp thresholds? Study monotone increasing symmetric functions v can we show they typically have sharp thresholds?}

\subsection{Entropy maximisation in the decorated graph setting}
Recall that the \emph{discrete entropy} or \emph{Shannon entropy} of a random variable $X$ taking values inside a discrete set $S$ is $\sum_{s\in s} -\mathbb{P}(X=s)\log\left( \mathbb{P}(X=s)\right)$. The entropy we consider in this paper (see Definition~\ref{definition: volume, entropy density}) can be viewed as a continuous analogue of discrete entropy when $X$ is a point sampled uniformly at random from some body $b\in \mathbb{R}^n$.

In the $\{0,1\}$-decorated setting, the rough structure of discrete entropy maximisers for hereditary properties of graphs is well-understood, via the choice number~$\chi_c$ (see~\cite{Alekseev93,BollobasThomason95} for the  set of possible `entropy densities' $\pi(\mathcal{P})$ and~\cite{AlonBaloghBollobasMorris11} for the possible structure of entropy maximisers).  By contrast, it is less clear what the set of possible values of entropy densities or the possible rough structure of graphs maximising entropy should be in the $[k]$-decorated setting for $k\geq 3$, let alone the set of entropy maximisers in the setting of $[0,1]$-decorated graphs.  We are only aware of one partial result in this area: Alekseev and Sorochan~\cite{AlekseevSorochan:colored} who established a dichotomy on the growth rate of a symmetric hereditary property of $[k]$-decorations of $E(K_n)$.  Moreover, it is clear that the possible structures of entropy maximisers are much more varied than in the case~$k  = 2$, see the discussion at the end of~\cite{FalgasRavryOconnellUzzell18}. This leads to the following analytic problems.
\begin{problem}\label{problem: possible maximum entropy in multicolour setting} 
	Let $k\in \N$ with $k\geq 3$. Let $\PP$ be a hereditary property of $[k]$-decorations of $E(K_n)$ and $\limitsx{\PP}$ be its completion under the cut norm. Determine the set of possible values for $m(\PP)=\sup_{W\in \limitsx{\PP}} \mathrm{Ent}(W)$, as well as the possible structures of entropy maximisers. 
\end{problem}
\begin{problem}\label{problem: possible maximum entropy in [0,1] setting}
Let $\PP$ be a hereditary property of $[0,1]$-decorated graphs and $\limitsx{\PP}$ be its completion under the cut norm. Determine the set of possible values for $m(\PP)=\sup_{W\in \limitsx{\PP}} \mathrm{Ent}(W)$, as well as the possible structures of entropy maximisers. 
\end{problem}
\section*{Acknowledgements}
This work has a slightly tortuous history, which began when AU visited VFR and JS at Vanderbilt University in Fall 2015. VFR gratefully acknowledges the support of an AMS--Simons travel grant, which funded this visit. An (extended) first version of the graph limit part of this paper appeared on arXiv in~\cite[Sections 5--7]{FalgasRavryOConnellStrombergUzzell16} in July 2016. It was later decided by the authors of~\cite{FalgasRavryOConnellStrombergUzzell16} to split that unwieldy paper into two, one of which resulted in the prequel~\cite{FalgasRavryOconnellUzzell18} to this paper, and the other of which was a hoped-for extension of the graph limit part of~\cite{FalgasRavryOConnellStrombergUzzell16} to $[0,1]$-decorated graphons. Ultimately, due to a combination of a variety of life events and unexpected mathematical difficulties, the hoped-for extension did not materialise, but in 2019 VFR with the help of RH developed the geometric approximation results Theorems~\ref{theorem: containers for hereditary bodies} and Corollary~\ref{corollary: counting for ssee} and combined them with the graph limit result Theorem~\ref{theorem: maximum entropy gives growth rate for [k]} to form the present paper.
Finally, the authors would like to thank the two anonymous referees for their careful and helpful reviews.

\end{document}